\documentclass[11pt]{amsart}
\usepackage{amsfonts,amsmath,amsthm,amssymb,
color,amscd}
\usepackage{fullpage}

\numberwithin{equation}{section}

\newtheorem{thm}{Theorem}[section]
\newtheorem*{thm*}{Theorem}
\newtheorem{prop}[thm]{Proposition}
\newtheorem*{prop*}{Proposition}
\newtheorem{question}[thm]{Question}

\newtheorem{cor}[thm]{Corollary}

\newtheorem{defin}[thm]{Definition}

\newtheorem{lemma}[thm]{Lemma}

\newtheorem{example}[thm]{Example}
\newtheorem{remark}[thm]{Remark}
\newtheorem*{remark*}{Remark}
\newtheorem*{remarks*}{Remarks}

\newcommand{\ip}[1]{\langle #1 \rangle}

\newcommand{\GAC}{
	\begin{quote}
		\textbf{Generalized Alekseevsky Conjecture.}  Let $M = G/H$ be a homogeneous Ricci soliton with $c<0$.  Then $M$ is isometric to a simply-connected solvable Lie group with left-invariant solvsoliton metric.
\end{quote}
}

\begin{document}

\title[Strongly solvable spaces]{Strongly solvable spaces}
\author[Michael Jablonski]{Michael Jablonski
}
\thanks{
    This work was supported in part by NSF grant DMS-1105647.}
\date{March 6, 2014}
%Keywords: Homogeneous Ricci soliton; Einstein metric; isometry group; solvmanifold; solvable Lie group; almost completely solvable; generalized Alekseevskii conjecture

\begin{abstract}  This work builds on the foundation laid by Gordon and Wilson in the study of  isometry groups of solvmanifolds, i.e. Riemannian manifolds admitting a transitive solvable group of isometries.  We restrict ourselves to a natural class of solvable Lie groups called almost completely solvable; this class includes the completely solvable Lie groups.  When the commutator subalgebra contains the center, we have a complete description of the isometry group of any left-invariant metric using only metric Lie algebra information.

Using our work on the isometry group of such spaces, we study quotients of solvmanifolds.  Our first application  is to the classification of homogeneous Ricci soliton metrics.  We show that the verification of the  Generalized Alekseevsky Conjecture reduces to the simply-connected case.  Our second application is a generalization of a result of Heintze on the rigidity of existence of compact quotients for certain homogeneous spaces.  Heintze's result applies to spaces with negative curvature.  We remove all the geometric requirements, replacing them with algebraic requirements on the homogeneous structure.

\end{abstract}

\maketitle

\section{Introduction}
\label{sec: intro}
\subsection{Generalized Alekseevsky Conjecture}

The classification of non-compact, homogeneous Einstein metrics has been an actively pursued open problem for several decades.  In the 1970s, D. Alekseevskii conjectured that for any homogeneous space $G/K$ with $G$-invariant Einstein, the group $K$  is the maximal compact subgroup of $G$.  In particular, $G/K$ would be diffeomorphic to $\mathbb R^n$ and, if the group $G$ were linear, then $G/K$ would be isometric to a simply-connected, solvable Lie group with left-invariant metric.

To date the only known examples of non-compact, homogeneous Einstein metrics are isometric to solvable Lie groups with left-invariant metrics.  In the special case of solvable Lie groups, significant progress has been made and the classification  is well-understood from the works of Heber \cite{Heber} and Lauret \cite{LauretStandard}.

As it turns out, the techniques used to understand Einstein metrics on solvable Lie groups have proven extremely useful in the classification of a more general kind of metric, Ricci solitons, see \cite{Lauret:SolSolitons,Jablo:HomogeneousRicciSolitons}.  Recall, a Riemannian manifold $(M,g)$ is said to be Ricci soliton if there exists a smooth vector field $X$ on $M$ and constant $c\in \mathbb R$ such that $ric = cg + \mathcal L_Xg$.  Note, Einstein metrics are Ricci solitons with $X=0$.  As in the Einstein case, the only known examples of expanding homogeneous Ricci solitons (i.e.  $c<0$) are isometric to simply-connected solvable Lie groups with left-invariant metrics.  And so the following conjecture is believed by many to be true.
\GAC
In the conclusion of this conjecture, the Ricci soliton is expected to be a so-called solvsoliton, i.e.~a very special Ricci soliton with nice algebraic properties; see \cite{Jablo:HomogeneousRicciSolitons} or \cite{Lauret:SolSolitons}   for more information on solvsolitons.

One approach   towards the classification of non-compact, homogeneous Ricci solitons is the following   two-pronged strategy:    1) classify all such simply-connected spaces and 2) classify the quotients of the simply-connected spaces obtained in 1).  Very recently, Lafuente and Lauret \cite{LauretLafuente:StructureOfHomogeneousRicciSolitonsAndTheAlekseevskiiConjecture} have introduced a program for working on the simply-connected case and the results obtained reinforce the belief that the above conjecture is true.  Here we approach   2).

\begin{thm}\label{thm: strong gen alek conj reduced to s.c. case}
To verify the  Generalized Alekseevsky Conjecture, it suffices to verify it among simply-connected homogeneous manifolds.

More precisely, let $G/H$ be a homogeneous Ricci soliton which is covered by a (simply-connected) solvable Riemannian Lie group $S$.  Then $G/H$ is isometric to $S$, i.e. the covering is trivial, and $S$ may be chosen to be a solvsoliton.
\end{thm}

\begin{remark} The fact that the Ricci soliton metric above may be chosen to be a solvsoliton, after changing the solvable group, is proven in \cite{Jablo:HomogeneousRicciSolitons}.  The new content of this theorem is on the homogeneous spaces that $S$ covers, and while there often are homogeneous quotients of $S$, the above shows that none can be Ricci soliton.
\end{remark}

The strategy for proving this result is as follows.  Consider $G/H$ and its simply-connected cover $\widetilde{G/H}$; by hypothesis $S = \widetilde{G/H}$, for some solvable group $S$.  As $G/H$ comes equipped with a $G$-invariant metric, there exists a cover $\overline G$ of $G$ which is a subgroup of isometries of $S = \widetilde{G/H}$.  Thus, to understand the homogeneous spaces covered by a solvmanifold $S$, it suffices to understand the transitive groups of isometries of $S$.

We say a space is \textit{strongly solvable} if every transitive group of isometries contains a transitive solvable subgroup.  Observe, one consequence of being strongly solvable is that every homogeneous Riemannian manifold modelled on a solvmanifold is necessarily a solvmanifold.  Thus, if we can show that $\overline G$ above contains a transitive solvable subgroup, then $G/H$ will be a solvmanifold.  (We warn the reader that in general this is not the case, see Example \ref{ex: non-solvmanifold quotient of solvmanifold}.)  

 Theorem \ref{thm: strong gen alek conj reduced to s.c. case} is now a consequence of  the following theorem together with the work \cite{Jablo:HomogeneousRicciSolitons}, see Section \ref{sec: homogeneous Ricci solitons}  for more details.

\begin{thm} Let $S$ be a simply-connected solvable Lie group with left-invariant Ricci soliton metric.  Then $Isom(S)$ is linear and $S$ is strongly solvable.
\end{thm}

As all known examples of expanding homogeneous Ricci solitons (i.e. $c<0$) are simply-connected solvable Lie groups, we now have the fact that all known examples of such spaces have linear isometry groups.  This reinforces our belief that the Generalized Alekseevskii Conjecture is true.  We note that there do exist simply-connected solvable Lie groups with left-invariant metrics whose isometry groups are not linear, see   Example \ref{ex: bad metric on H^2 x R}.

In the special case of Einstein metrics, it is possible to deduce the above theorem from the works  \cite{AlekseevskyCortes:IsomGrpsOfHomogQuatKahlerMnflds,LauretStandard}; however, to our knowledge, the  observations above are new even in that special case.

In addition to the above reduction to the simply-connected setting, we address the following.

\begin{question}\label{question:when is homgo RS solvsoliton} If one had a homogeneous Ricci soliton at hand, is it possible to know if it is isometric to a solvsoliton?  Further, is there a simple procedure for accomplishing this?
\end{question}

Until the Generalized Alekseevsky Conjecture is proven, this question will remain important.  Using Proposition \ref{prop: test for isometric to solv}, we have such a procedure for determining when a given homogeneous space $G/H$ is locally isometric to a simply-connected strongly solvable Lie group. As all solvsolitons are such spaces by the above theorem, we have an affirmative answer to Question \ref{question:when is homgo RS solvsoliton}.

\subsection{General motivations}  Although much of our work could be motivated just by the application of classifying homogeneous Ricci solitons, we are further motivated by the following  general question:
	\begin{question}
	Given a homogeneous space
$G/H$ with a $G$-invariant metric, how much of the geometry of this Riemannian homogeneous space can be determined from local data? 
	\end{question}
 When $G/H$ is simply-connected, the Riemannian manifold can be completely reconstructed from a single tangent space $\mathfrak g/\mathfrak h$ with inner product.  Consequently,  all the geometry is encoded in a single tangent space $\mathfrak g/\mathfrak h$ with inner product.   While this is true in principle, in practice it is very difficult to recover much more than the equations for curvature from so little data.  

In this work, we are primarily interested in Riemannian solvmanifolds, i.e. homogeneous spaces with a transitive solvable group of isometries.  In the simply-connected case, such a space is isometric to a solvable Lie group with left-invariant metric.  We parse the question above into a more specific set of questions.   Let $S$ denote a solvable Lie group with left-invariant metric and metric Lie algebra $\mathfrak s$.  

\begin{enumerate}
	\item What does the algebraic information about $\mathfrak s$ tell us about $Isom(S)$?  Can we reconstruct the isometry group using information about the given metric Lie algebra?  

	\item What are the possible alternative homogeneous structures on $S$? I.e., what are the transitive groups of isometries of $S$?   

	\item Is there a procedure for determining when $G/H$ is isometric to a solvmanifold $S$?  

	\item 	What kind of (homogeneous) spaces are covered by $S$?  Can one say when there exists a metric which admits a compact quotient covered by $S$?
\end{enumerate}

By and large, answering the last three questions reduces to a thorough understanding of the isometry group of the given space.  This is where we start.

\subsection{Isometry groups of solvable Lie groups with left-invariant metrics}
Let $S$ be a simply-connected solvable Lie group with left-invariant metric $g$.  Denote the Lie algebra of $S$ by $\mathfrak s$ and the restriction of $g$ to $\mathfrak s$ by $g$.  Our goal is to reconstruct the full isometry group using only information coming from the metric Lie algebra $\{\mathfrak s,g\}$.

As $S$ is simply-connected, we have that $Aut(S) \simeq Aut(\mathfrak s)$ and the following group is a group of isometries
	$$(Aut(\mathfrak s) \cap O(g)) \ltimes S$$
Such a group of isometries is called the \textit{algebraic isometry group}.  

If $S$ is nilpotent or, more generally, completely solvable and unimodular, then this group of isometries is the full isometry group, see \cite{GordonWilson:IsomGrpsOfRiemSolv}.  In these cases, the isometry group is completely understood in terms of, and recoverable from, metric Lie algebra data.  

For a general homogeneous Riemannian manifold, very little in known about how the isometry group relates to the given homogeneous presentation.  In the presence of a transitive reductive group of isometries, progress has been made \cite{Gordon:RiemannianIsometryGroupsContainingTransitiveReductiveSubgroups}.  The only other case where progress has been made occurs when there exists a unimodular solvable group acting.  In \cite{Gordon:RiemannianIsometryGroupsContainingTransitiveReductiveSubgroups}, a procedure for reconstructing the isometry group is given;  
see Section \ref{sec: strongly solv spaces}  for more details.

In contrast to the above, the case that $S$ is non-unimodular is considerably more complicated (compare the following example to Example \ref{ex: bad metric on H^2 x R}).  Take $S$ which is an Iwasawa subgroup of a semi-simple Lie group $G$ of non-compact type.  As $S$ acts simply-transitively on a symmetric space, there exists a left-invariant metric on $S$ so that the (connected) isometry group is $G$.  Any general technique for constructing the isometry group of a solvmanifold would have to incorporate these two very different kinds of isometry groups that can arise.

We restrict ourselves to a natural, and fairly large, class of solvable Lie groups.

\begin{defin} A solvable Lie algebra $\mathfrak s$ is called almost completely solvable if, for $X\in \mathfrak s$, the following property holds: $ad\ X$ having only purely imaginary eigenvalues implies  $ad\ X$ has only the zero eigenvalue.  Moreover, we say a Lie group $S$ is almost completely solvable if its Lie algebra 
is so.
\end{defin}

This class of groups includes many familiar classes of Lie groups, including
\begin{itemize}
	\item Nilpotent Lie groups,
	\item More generally, completely solvable groups, i.e. when $ad~X$ has only real eigenvalues,
	\item Any solvable Lie group admitting a metric of negative curvature \cite{HeintzeHMNC},
	\item Any solvable Lie group admitting an Einstein metric \cite{LauretStandard}.
\end{itemize}

Note, the last two examples are not contained in the set of completely solvable groups.  To see that these last two examples are in the class of almost completely solvable groups   see Propositions \ref{prop: neg curv alg is almost compl solv} and \ref{prop: einstein solv is almost compl solv}.

\begin{defin} We say a Lie algebra $\mathfrak s$ is admissible if $\mathfrak z(\mathfrak s) \subset [\mathfrak s,\mathfrak s]$.  Moreover, we say a Lie group $S$ admissible if its Lie algebra is so.
\end{defin}

In the above list of examples, the last two examples are also examples of admissible Lie groups.

\begin{thm}\label{thm: admissible almost compl solv has linear isometry group and reconstructable from metric lie algebra data}
	Let $S$ be a simply-connected, admissible, almost completely solvable Lie group.  Given any left-invariant metric $g$ on $S$, the isometry group of $S$ is linear and, furthermore, can be constructed using only metric Lie algebra data.
\end{thm}

Our procedure for constructing the isometry group of a solvable group as above is inspired by the work \cite{AlekseevskyCortes:IsomGrpsOfHomogQuatKahlerMnflds} where very special metric solvable Lie algebras are studied.  Our goal has been to remove any a priori metric requirements and replace them with purely algebraic constraints.  

We give a brief idea for how one obtains the isometry group of an admissible solvable Lie group here; see Section \ref{sec: Algorithm for full isometry group} for full details on reconstructing the isometry group for these groups.  
  The idea is to split apart our solvable Lie algebra $\mathfrak s = \mathfrak s_1 + \mathfrak s_2$ into a so-called \textit{LR-decomposition} where   $\mathfrak s_1$ is a subalgebra and $\mathfrak s_2$ is an ideal.  The first factor $\mathfrak s_1$ contributes a semi-simple group of isometries (as in the case of a symmetric space) and the second factor contributes algebraic isometries (as in the nilpotent or unimodular case).  By picking $\mathfrak s_1$ to be `maximal', one obtains the full isometry group.

\begin{remark} Even in the case of having a completely solvable Lie group, if one does not have the property of being admissible, it is very difficult to understand the isometry group from metric Lie algebra data.  See Example \ref{ex: bad metric on H^2 x R}.
\end{remark}

\subsection{Rigidity of existence of compact and finite volume quotients}  We also apply our work to the existence of compact quotients covered by solvmanifolds.
\begin{question}
Given a homogeneous space $G/H$, when does there exists a compact Riemannian manifold covered by $G/H$?  
\end{question}
It is well-known that $Isom(G/H)$ must be unimodular, but very little is known about the algebraic constraints placed on $G$ and $H$ to guarantee or preclude the existence of such a metric.

When $G/H$ is a nilmanifold $N$, the existence of such a metric is a purely algebraic problem.  More precisely, such $N$ admits a metric with compact quotient if and only if $\mathfrak n$ admits a rational structure.  This is known as Malcev's criterion.  Here $N$ actually admits a lattice itself and  the existence of one  metric with compact quotient on a nilmanifold implies every metric on that manifold admits a compact quotient.

Outside the realm of nilmanifolds, the problem is much more subtle.  Consider the familiar example of a solvable Lie group $S$ which is the Iwasawa subgroup of a semi-simple group $G$ of non-compact type.  While $S$ does not admit a lattice (as it is non-unimodular), we can endow it with the symmetric metric and so obtain a metric with compact quotients.  As we will see, the symmetric metric is the only one admitting a compact quotient.

\begin{question}  Are there metric or algebraic criteria that guarantee or preclude the existence of homogeneous metric on $G/H$ which admits a compact or finite volume quotient?
\end{question}

We produce such a criterion for a certain class of homogeneous spaces.  Given a Lie algebra $\mathfrak s$, denote by $\mathfrak{n(s)}$ the nilradical.

\begin{defin}  We say a solvable Lie algebra $\mathfrak s$ is positive if there exists $X\in \mathfrak s$ such that the  
eigenvalues of $ad~X|_{\mathfrak{n(s) } } $ have positive real part. Such an element $X$ will be referred to as a positive element of $\mathfrak s$.  We say a Lie group $S$ is positive if its Lie algebra is so.
\end{defin}

\begin{thm}\label{thm: general finite volume quotient is symmetric}  Let $S$ be a simply-connected, almost completely solvable, positive Lie group with left-invariant metric $g$.   The following are equivalent
	\begin{enumerate}
		\item $\{S,g\}$ has a compact quotient (as a Riemannian manifold).
		\item $\{S,g\}$ has a finite volume quotient.
		\item The full isometry group of $\{S,g\}$ is unimodular.
		\item $\{S,g\}$ is symmetric
	\end{enumerate}	
\end{thm}

This result was proven in a special case by Alekseevsky-Cortes \cite{AlekseevskyCortes:IsomGrpsOfHomogQuatKahlerMnflds}.  There the algebras $\mathfrak s$ are completely solvable, $\mathfrak s = \mathfrak a + \mathfrak{n(s)}$ with $ad~\mathfrak a$ abelian and fully reducible, and there are somewhat strong conditions on the metric $g$. e.g. $\mathfrak a \perp \mathfrak{n(s)}$.

\begin{cor}\label{cor: negative curv rigidity} Let $G/H$ be a homogeneous space which admits a $G$-invariant metric with negative curvature.  Now consider any $G$-invariant metric $g$ on $G/H$.  The Riemannian homogeneous space $\{G/H , g\}$ admits a finite volume quotient if and only if it is symmetric.
\end{cor}

In the special case that the metric $g$ has negative curvature, this result was proven by Heintze \cite{Heintze:CompactQuotientsOfHomogNegCurvRiemMflds}.

\begin{cor}\label{cor: einstein alg rigidity} Let $S$ be solvable Lie group which admits an Einstein metric.  Consider a left-invariant metric $g$ on $S$.  The Riemannian homogeneous space $\{S,g\}$ admits a finite volume quotient if and only if it is symmetric.
\end{cor}

In the special case that the metric $g$ is actually Einstein, this follows from the classification results of Lauret \cite{LauretStandard} together with the work of Alekseevsky-Cortes \cite{AlekseevskyCortes:IsomGrpsOfHomogQuatKahlerMnflds}.
\\

\textit{Acknowledgments.}  It is with pleasure that we thank Carolyn Gordon for helpful comments on an early version of this work.  We are especially grateful to her for providing us with the illustrative Example \ref{ex: non-solvmanifold quotient of solvmanifold}.

\subsection{Organization of the paper}
In Sections \ref{sec:prereq linear} and \ref{sec:prereq isom grp solvmfld}, we discuss prerequisites on linear Lie algebras and isometry groups of solvmanifolds, respectively.  In Section \ref{sec: strongly solv spaces} we discuss strongly solvable spaces.  In Section \ref{sec:proof of main theorems}, we present the proofs of our main results on the isometry groups of admissible, almost completely solvable spaces.  In Section \ref{sec: Algorithm for full isometry group}, we produce the algorithm for computing the full isometry group of such solvmanifolds.  Sections \ref{sec: homogeneous Ricci solitons} and \ref{sec: applications to compact quot of solv} are dedicated to the applications of classifying  homogeneous Ricci solitons and rigidity of existence of compact quotients, respectively.

\section{Prerequisites for linear Lie algebras}
\label{sec:prereq linear}
\subsection{Linear Lie algebras}

A Lie group $G$ is called {linear} if it is a subgroup of $GL(V)$, for some vector space $V$, and {linearizable} if there exists a faithful representation $\rho : G\to GL(V)$.  Furthermore, a Lie group $G<GL(V)$ is said to be fully reducible if given any invariant subspace of $V$, there exists an invariant complementary subspace in $V$.  Analogously, we define fully reducible Lie algebras.

Let $\mathfrak g \subset \mathfrak{gl}(V)$ be the Lie algebra of a fully reducible group $G<GL(V)$, then the algebra $\mathfrak g$ is fully reducible.  An example that will appear repeatedly in the sequel is when the group $G$ is compact.

Given a linear Lie algebra $\mathfrak g$, we say a subalgebra $\mathfrak h \subset \mathfrak g$ is a maximal fully reducible  subalgebra if it is not a proper subalgebra of a fully reducible subalgebra of $\mathfrak g$.

\begin{thm}[Mostow]\label{thm: mostow - mfr are conjugate}
Maximal fully reducible subalgebras of linear Lie algebras are conjugate via inner automorphisms coming from the nilradical.
\end{thm}

For a proof, see \cite{Mostow:FullyReducibleSubgrpsOfAlgGrps}.  As maximal fully reducible subalgebras are essentially unique, we denote a choice of such an algebra of $\mathfrak g$ by  $mfr(\mathfrak g)$.  
For a proof of the following, see Corollary 4.1, loc.~cit.  

\begin{lemma}\label{lemma: levi factor commutes with mfr(g2)}
Take a linear Lie algebra $\mathfrak h$ with Levi decomposition $\mathfrak h = \mathfrak h_1\ltimes \mathfrak h_2$.
	$$mfr (\mathfrak h) = \mathfrak h_1 + mfr(\mathfrak h_2)$$
for some maximal fully reducible subalgebra $mfr(\mathfrak h_2)$ of $\mathfrak h_2$.
\end{lemma}

In practice, to select the maximal fully reducible subalgebra of a solvable algebra, it will be useful to apply a strong version of Lie's theorem.

\begin{thm}[Lie's theorem]\label{Lie's theorem} Let $\mathfrak s$ be a solvable subalgebra of $\mathfrak{gl}(V)$, where $V$ is a complex vector space.  There exists a basis of $V$ such that
	\begin{enumerate}
		\item $\mathfrak s$ is a subalgebra of upper triangular matrices,
		\item $mfr(\mathfrak s)$ is a subalgebra of diagonal matrices,
		\item the nilradical $\mathfrak{n(s)} = \mathfrak n_1 + \mathfrak n_2$, where $\mathfrak n_1$ is a subalgebra of strictly upper triangular matrices and $\mathfrak n_2$ is some complement of $[\mathfrak s,\mathfrak s] \cap \mathfrak z(\mathfrak s)$ in the center $\mathfrak z(\mathfrak s)$.
	\end{enumerate}
\end{thm}

\begin{proof}
The first condition is the classical form of Lie's theorem.  To prove (ii), one recognizes that the set of diagonal matrices is a maximal fully reducible subalgebra of the set of upper triangular matrices.  The result follows by applying Mostow's theorem on the conjugacy of maximal fully reducible subalgebras to the algebra of upper triangular matrices.

The proof of (iii) reduces to the observation that for a solvable Lie algebra, the nilradical is precisely the set of ad-nilpotent elements.
\end{proof}

\begin{cor}\label{cor:reductive elements of commutator of solvable}
	Let $\mathfrak s$ be a solvable subalgebra of $\mathfrak{gl}(V)$.  Then $Y\in [\mathfrak s,\mathfrak s]$ is fully reducible if and only if $Y=0$.
	\end{cor}

This result holds for real as well as complex vector spaces.  To see this, one can simply complexify a given real space and then apply the result over $\mathbb C$.

\begin{lemma}
\label{lemma: eigenvalues of ad X = those of ad X+N for N nilpotent}
Consider a solvable Lie algebra $\mathfrak s$ with nilradical $\mathfrak{n(s)}$.  For $X\in\mathfrak s$, the  eigenvalues of $ad~X$ are the same as those of $ad~(X+N) = ad~X + ad~N$ for any $N\in\mathfrak{n(s)}$.
\end{lemma}

\begin{proof}
To prove this, we apply Lie's Theorem to the solvable algebra $ad~\mathfrak s^\mathbb C$, where $\mathfrak s^\mathbb C$ is the complexification of $\mathfrak s$.  By the theorem above, this subalgebra of $\mathfrak{gl}(\mathfrak s^\mathbb C)$ may be considered as a subalgebra of upper triangular matrices with $ad~\mathfrak{n(s)}$ in the set of strictly upper triangular matrices.  As the zeros of $|ad~X - \lambda Id|$ are precisely the diagonal entries, it is now clear that $ad~X$ and $ad~X + ad~N$ have the same eigenvalues.
\end{proof}

\begin{remark}  For the above, we note that multiplicities of the eigenvalues do change, in general.
\end{remark}

\begin{prop}\label{prop: neg curv alg is almost compl solv}
	Let $\mathfrak s$ be a solvable Lie algebra such that the corresponding simply-connected Lie group $S$ admits a left-invariant metric of negative curvature.  Then $\mathfrak s$ is almost completely solvable.
\end{prop}

\begin{proof}
From \cite{HeintzeHMNC}, it is known that $\mathfrak s = \mathfrak a + \mathfrak{n(s)}$, where $\dim \mathfrak a = 1$ and $X\in\mathfrak a$ satisfies the property that the (generalized) eigenvalues of $ad~X$ have positive real part.  By the lemma above, for any $X\in\mathfrak a - \{0\}$ and $N\in\mathfrak{n(s)}$, $ad~(X+N)$ has (generalized) eigenvalues with positive real part.  Thus, the only elements $X\in\mathfrak s$ so that $ad~X$ has purely imaginary eigenvalues are the elements of the nilradical.  This proves the claim.
\end{proof}

\begin{prop}\label{prop: einstein solv is almost compl solv}
	Let $\mathfrak s$ be a solvable Lie algebra such that the corresponding simply-connected Lie group $S$ admits a left-invariant Einstein metric.  Then $\mathfrak s$ is almost completely solvable.
\end{prop}

\begin{proof}
	By the classification results of Lauret \cite{LauretStandard}, such a solvable Lie algebra can be decomposed as $\mathfrak s = \mathfrak a + \mathfrak{n(s)}$ where $\mathfrak a$ is an abelian Lie algebra such that $ad~\mathfrak a$ is fully reducible and no element has purely imaginary eigenvalues.  The proof of our proposition is finished as in the proof of the previous proposition.
\end{proof}

\subsection{Linearizing Lie groups}
\label{sec: linearizing Lie groups}
In Section \ref{sec:proof of main theorems}, we  apply Mostow's result on maximal fully reducible subalgebras to the Lie algebra $\mathfrak g$ of the isometry group of a solvmanifold.  To use those results, we first need an almost faithful representation of the isometry group $G$ of a solvmanifold.  We denote by $Rad(G,G)$ the radical of the derived group $(G,G)$ of $G$.

\begin{prop}\label{prop: existence of almost faithful representation}
	Let $G$ be a Lie group with finitely many connected components.  Then there exists an almost faithful representation of $G$ if and only if $Rad(G,G)$ is simply-connected.
	\end{prop}

We begin by showing the existence of an almost-faithful representation when $Rad(G,G)$ is simply-connected.  As we will demonstrate, in the case that the Lie group is connected, the proof of this result follows immeadiately from two classical results of Malcev and Goto on the linearization of Lie groups, \cite{Malcev:OnLinearLieGroups,
Goto:FaithfulRepresentationsOfLieGroupsII}.  The reduction to the case that $G$ is connected is Proposition 5.3 of \cite{GorbatsevichOnishchikVinberg:LieGroupsAndLieAlgebras3}.

\begin{thm}[Malcev, 1943] Let $G$ be a connected Lie group with Levi factor $L$ and radical $Rad(G)$.  Then $G$ is linearlizable if and only if $L$ and $Rad(G)$ are so.
\end{thm}

\begin{thm}[Goto, 1950]\label{thm: Goto}  A connected Lie group $G$ is linearizable  if and only if a Levi factor is linearizable and $Rad(G,G)$ is simply-connected.
\end{thm}

Consider a Levi decomposition $G=L \ Rad(G)$.  To build an almost faithful representation of $G$, we start by replacing $L$ with a suitable quotient.  Denote by $\psi : L \to Aut (Rad(G))$ the restriction of conjugation by $L$ in $G$ to the radical.

\begin{lemma} There exists a discrete central subgroup $Z$ of $L$ such that 
	\begin{enumerate}
	\item $Z < Ker \ \psi$
	\item $L/Z$ is linearizable
\end{enumerate}	 
\end{lemma}

\begin{proof}
The image $\psi (L) < Aut (Rad (G)) < Aut(Rad(\mathfrak g))$ is a semi-simple subgroup of a linear group.  
As $Ker~\psi$ is a normal subgroup of $L$, we may decompose $L = L_1 \dots L_k$ as a product of simple factors  so such that $\psi(L)$ is locally isomorphic to $L^1:=L_1\dots L_i$, with $L^2:=L_{i+1}\dots L_k < Ker \ \psi$.  

Denote  the center of $L^2$ by $Z(L^2)$.  Using $Z = Ker \ (\psi|_{L^1}) \cdot Z(L^2)$, we have the desired result. 
\end{proof}

\begin{remark} One does not know  if $Z$ is the full center of $L$.  However, it is a finite index subgroup.
\end{remark}

To find an almost faithful representation of $G$, it suffices to find a faithful representation of $H=L/Z \cdot Rad(G)$.  By our choice of $Z$ above, we see that $Rad(H,H)=Rad(G,G)$.  And so, by the theorem of Goto above, $H$ is linearizable precisely when $Rad(G,G)$ is simply-connected.  This proves one direction of our proposition.

To finish, we prove the reverse implication.  Let $\psi$ be an almost faithful representation of $G$.  Observe that $\psi (Rad(G,G)) = Rad(\psi (G),\psi (G))$.  
Furthermore, observe that linear nilpotent Lie groups must always be simply-connected, as they are conjugate to upper triangular matrix groups (Lie's theorem).  As $\psi$ has discrete kernel, we see that the quotient $\psi (Rad(G,G))$ of $Rad(G,G)$ being simply-connected implies that $Rad(G,G)$ is simply-connected.

\subsection{Representations of semi-simple groups}
Let $G$ be a semi-simple group of non-compact type and $\rho$ an almost faithful representation.  Consider an Iwasawa decomposition $G=KAN$ and $M = N_K(A)$. 

\begin{lemma}\label{lemma: rho M and A} $\rho(M)$ is compact and $\rho(A)$ is fully reducible and diffeomorphic to some euclidean space.
\end{lemma}

This lemma follows immediately from the fact that $\rho(K)\rho(A)\rho(N)$ is an Iwasawa decomposition of the linear semi-simple group $\rho(G)$.

\section{Prerequisites for isometry group of solvmanifold}
\label{sec:prereq isom grp solvmfld}

To study the isometry groups of the solvmanifolds of interest in this work, we make use of the very general results obtained by Gordon and Wilson \cite{GordonWilson:IsomGrpsOfRiemSolv}.  For ease of referencing through out our work, we record many of the needed results in this section.

The main result of that work is the  construction of a solvable group in so-called standard position.  This special solvable structure is unique, up to conjugation within the isometry group, and most impressively can be recovered from any other solvable structure via an algorithm that only utilizes metric Lie algebra data.  We present these results below and restrict ourselves to the simply-connected setting for simplicity.

Let $\{S,g\}$ be a simply-connected solvable Lie group with left-invariant metric.  Denote the connected isometry group by $G=Isom(S,g)_0$ and its Lie algebra by $\mathfrak g$.  Through out this section, and the paper as a whole, we will maintain consistent notation with Gordon-Wilson; for explanation of notation, we direct the interested reader to that work.

The following is from Definition 2.2 and Proposition 2.4 of \cite{GordonWilson:IsomGrpsOfRiemSolv}.
\begin{defin}\label{def: normal modification}  Let $\mathfrak s$ be a solvable subalgebra of $\mathfrak g$.  A normal modification of $\mathfrak s$ is a (solvable) subalgebra of $\mathfrak g$ of the form $(Id + \phi) (\mathfrak s)$ where $Id$ is the identity map on $\mathfrak s$ and
	\begin{enumerate}
	\item $\phi : \mathfrak s \to \mathfrak g$ is a linear map
	\item $\phi(\mathfrak s)$ is abelian and contained in a compactly embedded subalgebra $\mathfrak l \subset \mathfrak g$ such that $\mathfrak l \cap \mathfrak s = \{0\}$
	\item $\phi(\mathfrak s) \subset N_\mathfrak g(\mathfrak s)$, the normalizer of $\mathfrak s$ in $\mathfrak g$
	\item $[\mathfrak s,\mathfrak s] \subset Ker~\phi$	
	\end{enumerate}
\end{defin}

\begin{defin}\label{def: std. mod.} Let $\{\mathfrak s,g\}$ be a metric solvable Lie algebra.  Denote the algebra of skew-symmetric derivations of $\mathfrak s$ (relative to $g$) by $N_\mathfrak l(\mathfrak s)$.  The standard modification $\mathfrak s'$ of $\mathfrak s$ is defined to be the orthogonal complement of $N_\mathfrak l(\mathfrak s)$ in $\mathfrak f = N_\mathfrak l(\mathfrak s) \ltimes \mathfrak s$ relative to the Killing form $B_\mathfrak f$ of $\mathfrak f$.
\end{defin}

The following proposition records some useful properties of the standard modification.

\begin{prop}\label{prop: properties of std. mod.}   Consider a metric solvable Lie algebra $\{\mathfrak s,g\}$ with standard modification $\mathfrak s'$.
\begin{enumerate}
	\item $\mathfrak s'$ is a solvable ideal of $N_\mathfrak l (\mathfrak s) \ltimes \mathfrak s$,
	\item The standard modification is a normal modification.
\end{enumerate}
\end{prop}

\begin{proof}  For (i), see the remarks after Def.~3.4 of \cite{GordonWilson:IsomGrpsOfRiemSolv}.  For (ii), see the remarks following Definition 3.4, Definition 2.2 (iv), and  Proposition 2.4 (d), loc.~cit.
\end{proof}

\begin{defin} Let $\{\mathfrak s,g\}$ be a metric solvable Lie algebra with isometry algebra $\mathfrak g$ (i.e. the Lie algebra of the isometry group of $\{S,g\}$).  We say an algebra $\mathfrak s$ is in standard position if it equals its own standard modification $\mathfrak s'$.
\end{defin}

This is not the definition given in Section 1 of \cite{GordonWilson:IsomGrpsOfRiemSolv}, but it is easier to digest for the unfamiliar reader and  is equivalent by Corollary 3.7, loc.~cit.  From Section 3 of that work,  we have the following.

\begin{thm}\label{thm: GW - standard position}  Let $\mathfrak g$ be the isometry algebra of a solvmanifold with metric Lie algebra $\{\mathfrak s,g\}$.  
\begin{enumerate}
\item There exists a solvable subalgebra $\mathfrak s'' \subset \mathfrak g$ which is in standard position,
\item Up to conjugacy within $\mathfrak g$, there is only one solvable subalgebra in standard position,
\item Further, $\mathfrak s''$ may be obtained from $\mathfrak s$ by applying two standard modifications, successively.
\end{enumerate}
\end{thm}

As we can obtain a solvable algebra in standard position by applying two successive standard modifications, this is  called the standard modification algorithm.  Abusing language, we  often refer to `the' solvable group in standard position.  

Let $\mathfrak g = \mathfrak g_1 + \mathfrak g_2$ be a Levi decomposition (i.e. $\mathfrak g_1$ is semi-simple and $\mathfrak g_2$ is the radical of $\mathfrak g$).  Further, decompose the Levi factor $\mathfrak g_1 = \mathfrak g_{nc} + \mathfrak g_c$ into ideals which are of non-compact and compact type, respectively.  Consider an Iwasawa decomposition $\mathfrak g_{nc} = \mathfrak k + \mathfrak a +\mathfrak n$ and denote by $\mathfrak m = N_\mathfrak k(\mathfrak a)$ the normalizer of $\mathfrak a$ in $\mathfrak k$.  The following results come from Section 1 of \cite{GordonWilson:IsomGrpsOfRiemSolv}.

\begin{lemma}
\label{lemma: g2 in zNls + s}\label{lemma: nilrad(G) < S}\label{lemma: properties of f}
Let $\{S,g\}$ be a simply-connected solvable Lie group with left-invariant metric.  Denote the isometry algebra by $\mathfrak g$ and the isotropy algebra at a point $p\in S$ by $\mathfrak g_p$.  Consider $\mathfrak f = \mathfrak m + \mathfrak a + \mathfrak n + \mathfrak g_c + \mathfrak g_2$, then
\begin{enumerate}
\item For any choice of $p\in S$, the orthogonal complement of $\mathfrak f\cap \mathfrak g_p$ in $\mathfrak f$, relative to the Killing form $B_\mathfrak f$ of $\mathfrak f$, is a solvable algebra $\mathfrak s''$ in standard position,
\item $\mathfrak f = N_\mathfrak g(\mathfrak s'')$,
\item $\mathfrak f = N_\mathfrak l (\mathfrak s'') \ltimes \mathfrak s''$,
\item $\mathfrak g_2 \subset N_\mathfrak l (\mathfrak s'')\ltimes \mathfrak s''$,
\item $nilrad(\mathfrak g) < \mathfrak s''$, 
\end{enumerate}
where $N_\mathfrak l (\mathfrak s'')$ denotes the set of skew-symmetric derivations of $\mathfrak s''$. %this is relative to the p-inner product where $L=G_p$
\end{lemma}

We point out that the last two items apply to any solvable algebra in standard position as $\mathfrak g_2$ and $nilrad(\mathfrak g) = nilrad(\mathfrak g_2)$ do not change under conjugation in $\mathfrak g$.

\begin{lemma}\label{lemma: standard simply connected solv acts simply transitively} Let $M$ be a simply-connected solvmanifold with solvable group of isometries $S$ in standard position.  Then $S$ acts simply-transitively and is simply-connected itself.
\end{lemma}

There is no proof of this in in \cite{GordonWilson:IsomGrpsOfRiemSolv}, although it is alluded to.  The proof follows quickly from the construction of groups in standard position; we leave the proof to the diligent reader.

\begin{lemma}\label{lemma: open orbit implies transitive action} Let $M$ be a solvmanifold and $R$ a solvable subgroup of isometries such that the $R$ action on $M$ has an open orbit.  Then $R$ acts transitively on $M$
\end{lemma}

The case that $R$ is a closed subgroup of isometries by proven by Gordon-Wilson, see the proof of Theorem 1.11 (iii) in \cite{GordonWilson:IsomGrpsOfRiemSolv}.  For the general case, one simply observes that the open orbit of $R$ is again a homogeneous Riemannian manifold and so is complete.  However, an open complete submanifold of a complete manifold must be the entire manifold, i.e, $R$ acts transitively.

We warn the unfamiliar reader that this result is special to Riemannian homogeneous spaces.  There do exist topological homogeneous spaces $G/H$ with subgroups of $G$ acting with open, non-transitive orbits.

We record one final useful lemma.  For proof, see Lemma 2.6 of \cite{GordonWilson:IsomGrpsOfRiemSolv}.

\begin{lemma}\label{lemma: derivation of solvabe goes to nilradical} Let $\mathfrak s$ be a solvable Lie algebra with nilradical $\mathfrak{n(s)}$.  For any derivation $D\in Der(\mathfrak s)$, we have $D: \mathfrak s \to \mathfrak{n(s)}$.
\end{lemma}

\subsection{Almost completely solvable and standard position}
\label{sec: almost compl solv and std position}

\begin{thm}\label{thm: std. mod. preserves almost compl. solv.} Let $S$ be an almost completely solvable Lie group with left-invariant metric $g$.  Denote the isometry group by $Isom(S,g)$.  The solvable subgroup of $Isom(S,g)$ in standard position is also almost completely solvable.
\end{thm}

\begin{proof}
Recall, we can obtain the solvable subgroup in standard position by applying the standard modification algorithm twice, see Theorem \ref{thm: GW - standard position}.  To prove the theorem, it then suffices to  prove that applying the standard modification to an almost completely solvable group preserves this property.

To obtain the standard modification of $\mathfrak s$, we consider the set $N_\mathfrak l (\mathfrak s)$ of skew-symmetric derivations of $\mathfrak s$ (relative to $g$) and the algebra
	$$\mathfrak h := N_\mathfrak l (\mathfrak s) \ltimes \mathfrak s$$
The standard modification $\mathfrak s'$ is the orthogonal complement to $N_\mathfrak l (\mathfrak s)$ in $\mathfrak h$ relative to the Killing form of $\mathfrak h$.  As remarked by Gordon-Wilson (see Proposition \ref{prop: properties of std. mod.}), $\mathfrak s'$ is a solvable ideal of $\mathfrak h$.  Thus, $\mathfrak s'$ is contained in the radical of $\mathfrak h$.

Decompose $\mathfrak h = \mathfrak {lf} + \mathfrak r$ into a Levi factor $\mathfrak {lf}$ and the radical $\mathfrak r$.  Recall, the Levi factor is a maximal semi-simple subalgebra and this algebra is orthogonal to $\mathfrak r$ under the Killing form $B_\mathfrak h$ of $\mathfrak h$.  Here the radical is $\mathfrak r = \mathfrak z (N_\mathfrak l (\mathfrak s)) \ltimes \mathfrak s$, where $\mathfrak z(N_\mathfrak l (\mathfrak s))$ is the center of $N_\mathfrak l (\mathfrak s)$.

Denoting the Killing form of $\mathfrak r$ by $B_\mathfrak r$, we have
	$$ B_\mathfrak h (X,Y) = B_\mathfrak r (X,Y) \quad \mbox{ for } X\in \mathfrak r, \ Y\in \mathfrak r    $$
and so we see that the standard modification $\mathfrak s'$ is the orthogonal complement of $\mathfrak z(N_\mathfrak l (\mathfrak s))$ in $\mathfrak r = \mathfrak z (N_\mathfrak l (\mathfrak s)) \ltimes \mathfrak s$ relative to the Killing form of $\mathfrak r$.

Now consider the complexification of the adjoint representation of the solvable algebra $\mathfrak r = \mathfrak z (N_\mathfrak l (\mathfrak s)) \ltimes \mathfrak s$,
	$$ad(\mathfrak r): \mathfrak r \otimes \mathbb C \to \mathfrak r \otimes \mathbb C$$
Notice, this representation restricts to both $\mathfrak s\otimes \mathbb C$ and $\mathfrak s'\otimes \mathbb C$, as these algebras are ideals.  By Lie's Theorem (Theorem \ref{Lie's theorem}), we may view $ad(\mathfrak r)$ as a subalgebra of upper triangular matrices (relative to some choice of basis).  Further, we may assume that $\mathfrak z (N_\mathfrak l (\mathfrak s))$ is contained in the set of diagonal matrices as this is a fully reducible subalgebra.

Take $X\in \mathfrak s'$ such that $ad\ X : \mathfrak s' \to \mathfrak s'$ has purely imaginary eigenvalues.  Observe that the eigenvalues of $ad\ X:\mathfrak r \to \mathfrak r$ are also purely imaginary as, $\mathfrak s'$ being an ideal, we have only included more zero eigenvalues.  Write $X = D + Y$ with $D\in \mathfrak z(N_\mathfrak l (\mathfrak s))$ and $Y\in \mathfrak s$. As $ad\ D$  is a diagonal matrix in $\mathfrak {gl}(\mathfrak r \otimes \mathbb C)$, we have that the eigenvalues of the upper triangular matrix $ad\ Y = ad\ X -ad\ D$ are the eigenvalues of $ad\ X$ minus the eigenvalues of $ad\ D$.  This gives that the eigenvalues of $ad\ Y$ acting on $\mathfrak r$ are purely imaginary and hence the eigenvalues of $ad\ Y$ acting on $\mathfrak s$ are purely imaginary.

However, $\mathfrak s$ being almost completely solvable, by hypothesis, implies the eigenvalues of $ad\ Y|_\mathfrak s$ are zero.  Thus $ad\ Y|_\mathfrak r$ is nilpotent and $Y$ belongs to the nilradical of $\mathfrak r$.  

As the nilradical is in the kernel of $B_\mathfrak r$ and $\mathfrak s'$ is orthogonal to $\mathfrak z(N_\mathfrak l(\mathfrak s))$ under $B_\mathfrak r$, by construction,  we have
	$$ 0= B_\mathfrak r (D,X) = B_\mathfrak r(D,D) = tr \ (D|_\mathfrak s)^2$$
which implies $D=0$ as $D$ is skew-symmetric.  Thus $ad\ X = ad\ Y$ has only the zero eigenvalue and $\mathfrak s'$ is almost completely solvable.

\end{proof}

In general, the property of being an admissible Lie algebra is not preserved under normal and standard modifications.  (To see this, add a factor of $\mathbb R$ to any admissible algebra and modify this factor to act skew-symmetrically on the admissible algebra.)  It is not even clear if the property of being admissible is preserved for almost completely solvable algebras.  However, for our purposes, the following is enough.

\begin{thm}\label{thm: admissible almost compl solv impies center in [r,r]}
	Let $\mathfrak s$ be an almost completely solvable Lie algebra and let $\mathfrak r = \mathfrak z(N_l(\mathfrak s)) \ltimes \mathfrak s$.  If $\mathfrak s$ is admissible, then the  modification $\mathfrak s''$ of $\mathfrak s$ in standard position satisfies $\mathfrak{z(s'')} \subset [\mathfrak r,\mathfrak r]$.
\end{thm}

\begin{proof}
	Recall, we may construct $\mathfrak s''$ from $\mathfrak s$ by applying two successive standard modifications.  	First we claim that
		$$\mathfrak{z(s')} \subset [\mathfrak s,\mathfrak s] + \mathfrak{z(s)}$$
where $\mathfrak s'$ is the standard modification of $\mathfrak s$.  Using this and the theorem above, we immediately have that $\mathfrak{z(s'')} \subset [\mathfrak s',\mathfrak s'] + \mathfrak{z(s')} \subset [\mathfrak s',\mathfrak s'] + [\mathfrak s,\mathfrak s] + \mathfrak{z(s)}$.  By hypothesis, $\mathfrak{z(s)} \subset [\mathfrak s,\mathfrak s]$, and the result follows from the fact that $\mathfrak s$ and $\mathfrak s'$ are subalgebras of $\mathfrak r$.

	Now we prove the first claim.  We begin by showing $\mathfrak{z(s')} \subset \mathfrak s$.  As $\mathfrak{s'}$ is a standard modification of $\mathfrak s$, given $X\in \mathfrak{z(s')} \subset \mathfrak{s'}$, there exists $Y\in\mathfrak s$ and $D\in \mathfrak z(N_l(\mathfrak s))$ such that $X = D+Y$.  As $ad~X$ has only the zero eigenvalue on $\mathfrak s'$ (as it is trivial), and  $\mathfrak s'$ is an ideal of $\mathfrak r$, we see that $ad~X$ on $\mathfrak r$ has zero as its sole eigenvalue.
	
	As $\mathfrak s$ is an ideal of $\mathfrak r$, we may restrict $ad~X = D+ad~Y$  to $\mathfrak s$ and this restriction, again, has zero as its sole eigenvalue.  However, the proof of the previous theorem shows (via Lie's Theorem) that the eigenvalues of $ad~Y$ are sums of the eigenvalues of $D$ and $ad~X$.  Thus $ad~Y$ has purely imaginary eigenvalues and $\mathfrak s$ being almost completely solvable implies these are all zero.  Now we see that the eigenvalues of $D$ restricted to $\mathfrak s$ are all zero, i.e., $D=0$ and thus $X=Y\in \mathfrak s$.
	
	Recall, the standard modification is a normal modification (see Proposition \ref{prop: properties of std. mod.}) and thus we have that $\mathfrak s'$ is spanned by $[\mathfrak s,\mathfrak s]$ and $\{D_1+Y_1, \dots, D_k +Y_k\}$, where $\{Y_1,\dots,Y_k\}$ spans a complement to $[\mathfrak s,\mathfrak s]$ in $\mathfrak s$ and $D_i\in \mathfrak z(N_l(\mathfrak s))$.	As $X\in \mathfrak{z(s')}$, we see that 
		$$ 0 = [D_1+Y_1, X] = D_1(X) + [Y_1,X]$$
From this, the fact that $D\in Der(\mathfrak s)$, and the fact that $\mathfrak s'$ is an ideal of $\mathfrak r$, it follows that $D_1^{4k} (X) \in \mathfrak{z(s')} \cap [\mathfrak s,\mathfrak s]$.  Note that $D_1^{4k}$ is a symmetric operator on $\mathfrak s$ with non-negative eigenvalues.  Let $\lambda$ be the highest (non-zero) eigenvalue of $D_1^4$ with component $X_\lambda$ in $X$.  Then we see that
	$$X_\lambda = \lim _{k\to \infty}  \frac{1}{\lambda^k} D_1^{4k}X \in \mathfrak{z(s')} \cap [\mathfrak s,\mathfrak s]     $$
By considering $X-X_\lambda$, we may see the next highest eigen-component is also an element of $\mathfrak{z(s')} \cap [\mathfrak s,\mathfrak s]$.  Running through all the non-zero eigenvalues of $D_1^4$, we see that $X$ may be decomposed as $X= X' + X''$ where $X' \in \mathfrak{z(s')} \cap [\mathfrak s,\mathfrak s]$ and $X'' \in Ker ~ D_1$.  Note that $X'' = X - X' \in \mathfrak{z(s')}$ and we may apply the same argument to $X''$ with $D_2 + Y_2$, etc., to achieve a decomposition 
	$$X = X_1 + X_2, \quad   \mbox{ where } X_1 \in \mathfrak{z(s')} \cap [\mathfrak s,\mathfrak s]  \mbox{ and   }  X_2 \in \mathfrak{z(s')} \cap \cap_{i=1}^k Ker~D_i  $$
The last condition shows precisely that $X_2\in \mathfrak{z(s)}$.

Finally, applying the above arguments to a basis of $\mathfrak{z(s')}$ proves our claim and hence the theorem.
\end{proof}

\section{Strongly solvable spaces}
\label{sec: strongly solv spaces}

One of the main goals of this work is to understand the transitive groups of isometries which act on a given  solvmanifold.  We begin with an illustrative example which shows that the isometry group can be  very hard to recover just from the solvable Lie group at hand.  This example also demonstrates that a Riemannian homogeneous space can have radically different homogeneous structures.

\begin{example}\label{ex: bad metric on H^2 x R}
Consider $G=\widetilde{SL_2\mathbb R}$ and $S = H^2 \times \mathbb R$, where $H^2$ is an Iwasawa subgroup of $SL_2\mathbb R$.  
There exist left-invariant metrics on $G$ and $S$ making them isometric as Riemannian manifolds.
\end{example}

\begin{proof}
To constuct such metrics on $G$ and $S$, consider $G=\widetilde{SL_2\mathbb R}$ with Iwasawa decomposition $KAN$.  Endow $G$ with a left-invariant metric which is also right $K$-invariant; this is possible as the center $Z$ of $G$ is contained in $K$ and $K/Z$ is compact.  As computed in \cite{Gordon:RiemannianIsometryGroupsContainingTransitiveReductiveSubgroups}, the isometry group of this Riemannian Lie group is 
	$$G\times K / \Delta Z$$
where $\Delta Z$ is the diagonal imbedding of $Z$ in $K\times K$.  Consequently, we see that the solvable group $S=AN\times K$ acts simply-transitively on $G$ and we now have an example of a Riemannian manifold with two very different Lie groups acting isometrically and simply-transitively.
\end{proof}

\begin{remark}The group $H^2$ is a solvable group which acts simply-transitively on hyperbolic 2-space and the group $S$ is completely solvable.
\end{remark}

We say a Riemannian solvmanifold is \textit{strongly solvable} if every transitive group of isometries contains a transitive solvable subgroup.  Observe, if a space is strongly solvable, then the only groups acting almost simply transitively are solvable groups.  

\begin{example}  Let $S$ be a unimodular solvable Lie group.  Using any left-invariant metric, $S$ is a strongly solvable space.
\end{example}

To prove the claim,  it suffices to prove this in the case that $S$ is simply-connected.  We begin with the following proposition which gives a criterion for determining when a space is strongly solvable.  See Corollary \ref{cor: almost compl solv admissible implies strongly solvable} for another application.

\begin{prop}\label{prop: linear isom implie strongly solvable} Let $S$ be a simply-connected solvable Lie group with left-invariant metric.  If $Isom(S)$ is linear, then $S$ is a strongly solvable space.
\end{prop}

\begin{proof}
	As $S$ acts simply-transitively,  for any point $p\in S$,  we can write $Isom(S) = S (Isom(S)_p)$, where $Isom(S)_p$ is the isotropy subgroup at $p$.  Further, $S$ being simply-connected implies that the intersection of these two subgroups is trivial and $Isom(S)_p$ is a maximal compact subgroup of $Isom(S)$.

	Let $G$ be a transitive group of isometries of $S$.  We may assume $G$ is connected and decompose it into a Levi decomposition $G=G_1G_2$.  Further, we may write	
		$$G_1 = G_{nc}G_c = ANKG_c$$
Here $G_{nc}$ and $G_c$ are the non-compact and compact factors of $G_1$, respectively, and $G_{nc} = KAN$ is an Iwasawa decomposition.  	As $Isom(S)$ is linear, we see that $G_1$ is a linear semi-simple group and hence  the subgroup $K$ is compact.  Thus, there exists some $p\in S$ such that $KG_c \subset Isom(S)_p$.  Now the solvable subgroup $ANG_2$ of $G$ has an open orbit and hence it acts transitively (see Lemma \ref{lemma: open orbit implies transitive action}),  as desired.
\end{proof}

\begin{prop}\label{prop: isom of unimod solv is linear}
The isometry group of a simply-connected, unimodular, solvable Lie group is linear.
\end{prop}

This proposition and the preceding one together show that unimodular solvmanifolds are strongly solvable.  Before proving the proposition, we prove the following general lemma which will be used later as well.

\begin{lemma}\label{lemma: Rad(G,G) is s.c.}  Let $G$ denote the connected isometry group of a simply-connected solvmanifold.  Then   $Rad(G,G)$ is simply-connected.
\end{lemma}

\begin{proof}
   Let $S''$ denote the subgroup of $G$ in standard position.  As our solvmanifold is simply-connected, so is $S''$ (Lemma \ref{lemma: standard simply connected solv acts simply transitively}).  Recall, that the group $Rad(G,G)$ is nilpotent (see page 91, Cor. 2, of \cite{Jacobson:LieAlgebras}) and so 
	$$Rad(G,G) < Nilradical(G) < S''$$
See Lemma \ref{lemma: nilrad(G) < S} for the second inclusion. 

Since every Lie group is diffeomorphic to $C \times \mathbb R^n$, where $C$ is a maximal compact subgroup, we see that a solvable group is simply-connected if and only if every subgroup is simply-connected.  Thus $Rad(G,G)$ is simply-connected.
\end{proof}

\begin{proof}[Proof of Prop.~\ref{prop: isom of unimod solv is linear}]  
As discussed in Section \ref{sec: linearizing Lie groups}, it suffices to linearize the connected isometry group $G = Isom(S)_0$.   By  the result of Goto  (Theorem  \ref{thm: Goto}) and the lemma above, to prove that $G$ is linear it suffice to prove linearity of the Levi factor of $G$. 

Let $S''$ denote the solvable group in standard position in $G$.  Since $S$ is unimodular, one knows precisely what the isometry group of $S$ is, namely,
	$$Aut(\mathfrak s'') \cap O(\mathfrak s) \ltimes S''$$
This is the main result of \cite[Section 4]{GordonWilson:IsomGrpsOfRiemSolv}.  Now we see that the Levi factor of $Isom(S)$ can be chosen to be a subgroup of $Aut(\mathfrak s'') \cap O(\mathfrak s)$, which is clearly linear. 

\end{proof}

\begin{remark}
	We point out to the unfamiliar reader that Gordon and Wilson have given a procedure for reconstructing the isometry group of a unimodular solvable Lie group with left-invariant metric.  In this case, one applies the standard modification algorithm twice and then computes the algebraic isometry group for this standard modification.  See Section 4 of \cite{GordonWilson:IsomGrpsOfRiemSolv}.
\end{remark}

\subsection{Homogeneous quotients of strongly solvable spaces}

\begin{prop}\label{prop: homog space covered by strong solv is solv} Let $G/H$ be a Riemannian homogeneous space whose universal cover is a strongly solvable space $S$.  Then $G/H$ is a solvmanifold.
\end{prop}

\begin{proof}
	Although the action of $G$ does not lift to the universal cover $S$, there exists a cover $\overline G$ of $G$ which acts on $S$ by isometries.  Since $G$ acts transitively on $G/H$, $\overline G$ acts transitively on $S$.  	By hypothesis, $\overline G$ contains a solvable subgroup $R$ which acts transitively on $S$.  However, $R$ also acts (almost effectively) by isometries on $G/H$.  Thus, $G/H$ is a solvmanifold.
\end{proof}

\subsection{Homogeneous quotients which are not solvmanifolds}
Here we present an example of a solvmanifold with a homogeneous quotient which is not a solvmanifold.  We are grateful to Carolyn Gordon for graciously providing this example to us.

\begin{example}\label{ex: non-solvmanifold quotient of solvmanifold}  There exists a Riemannian metric on $SL_2(\mathbb R) \times T^2$ such that it is not a solvmanifold but it is covered by a solvmanifold.
\end{example}

Consider the group $G=SL_2(\mathbb R) \times \mathbb R^2$.  As $\mathbb Z^2 \subset \mathbb R^2$ is a central subgroup, any left-invariant metric on $SL_2(\mathbb R) \times \mathbb R^2$ gives rise to a left-invariant metric on $SL_2(\mathbb R) \times T^2$.  We will construct a left-invariant metric on $SL_2(\mathbb R) \times \mathbb R^2$ whose isometry group contains a transitive solvable subgroup, but so that the isometry group of $SL_2(\mathbb R) \times T^2$ does not.

Endow $\mathfrak g = \mathfrak{sl}_2(\mathbb R)$ with an $Ad~SO(2)$-invariant inner product.  Thus the Cartan decomposition $\mathfrak k \oplus \mathfrak p$, into skew-symmetric and symmetric parts, is an orthogonal decomposition.   
Let $\{e_1,e_2\}$ be an orthonormal basis of $\mathbb R^2$ and $\{X_1,X_2\}$  an orthonormal basis of $\mathfrak p$ where  $X_2 = ad~U(X_1)$ for some $U\in \mathfrak{so}(2)$.  Now we extend our inner product so that the orthogonal complement of $\mathfrak{sl}_2(\mathbb R)$ is spanned by $\{X_1 + e_1, X_2+e_2\}$.  Denote our choice of inner product by $\ip{\cdot , \cdot }$. Observe that $\mathfrak k \perp \mathbb R^2$ while $\mathfrak p \not \perp \mathbb R^2$.

Let $A:\mathbb R^2 \to \mathbb R^2$ denote the skew-symmetric map sending $e_1 \mapsto -e_2$ and $e_2 \mapsto e_1$.  Then for $U\in \mathfrak k$ as above, we see that $D= ad~U + A$ is a skew-symmetric derivation of $\mathfrak g$, relative to $\ip{\cdot,\cdot}$, and $\mathfrak{h} = \mathbb R D \ltimes \mathfrak g$ is thus a subalgebra of the Lie algebra of the isometry group of $\{G,\ip{\cdot,\cdot} \}$.

\begin{lemma}  Denote by $H$ the connected group of isometries of $\{G,\ip{\cdot,\cdot} \}$ with Lie algebra $\mathfrak h$.  Then $H$ is the connected component of the identity of the isometry group.
\end{lemma}

\begin{proof}
	To show this result, we  apply \cite[Theorem 1.11]{GordonWilson:IsomGrpsOfRiemSolv} by finding a solvable subgroup in standard position.  First, we must show that $\{G,\ip{\cdot,\cdot}\}$ is indeed a solvmanifold.
	
	Let $\mathfrak b$ denote the subalgebra of $\mathfrak h$ spanned by $U-D$.  This algebra commutes with $\mathfrak{sl}_2(\mathbb R)$ and normalizes $\mathbb R^2$.  Denote the connected subgroup of $H$ with Lie algebra $\mathfrak b$ by $B$.  Further, let $SL_2(\mathbb R) = KS$ be an Iwasawa decomposition where $K$ is the subgroup with Lie algebra $\mathfrak k$.  Observing that the isotropy subgroup of $H$ is contained in $B\times K$ (which is isomorphic to $SO(2) \times SO(2)$), we see that the group $B\ltimes \mathbb R^2 \times S$ acts transitively , by isometries, on $G$ - i.e., $\{G,\ip{\cdot, \cdot}\}$ is a solvmanifold.	 Here the $B$ action on the $SL_2(\mathbb R)$ factor is just right multiplication by $K$ (after inverting).
	
	To finish, we need to find a subgroup in standard position and compute its normalizer in the isometry group.  As $B\ltimes \mathbb R^2 \times S$ acts transitively on $G$, it inherits  a left-invariant metric $\ip{\cdot,\cdot}$. (As these two spaces are isometric, we abuse notation and reuse the same symbol for the inner product on $\mathfrak b \ltimes \mathbb R^2 + \mathfrak s$.)  Observe that the metric Lie algebra $\{\mathfrak b\ltimes \mathbb R^2 + \mathfrak s,\ip{\cdot,\cdot}\}$ has no skew-symmetric derivations.  Therefore, by the standard modification algorithm \cite[Theorem 3.5]{GordonWilson:IsomGrpsOfRiemSolv}, the group $B\ltimes \mathbb R^2 \times S$ is in standard position in the isometry group of $\{G,\ip{\cdot,\cdot}\}$.  Furthermore, as $\{\mathfrak b\ltimes \mathbb R^2 + \mathfrak s,\ip{\cdot,\cdot}\} $ does not admit any skew-symmetric derivations, $B\ltimes \mathbb R^2 \times S$ is its own normalizer in the connected component of the isometry group of $\{G,\ip{\cdot, \cdot}\}$ and we see that $H$ is indeed the connected component of the isometry group by \cite[Theorem 1.11]{GordonWilson:IsomGrpsOfRiemSolv}.
	
\end{proof}

As the group $SL_2(\mathbb R) \times T^2 = G/\mathbb Z^2$ is covered by $G$, any transitive group action on $G/\mathbb Z^2$ can be lifted (by passing to a cover of the group acting) to an action on $G$.  Thus, if $G/\mathbb Z^2$ were a solvmanifold, then a cover of the solvable group would act on $G$.  Further, the group acting on $G$ would have to leave the set $\mathbb Z^2$ fixed.  While it is clear that the group $B\ltimes \mathbb R^2 \times S$ does not leave $\mathbb Z^2$ fixed, to ensure that $G/\mathbb Z^2$ is not a solvmanifold, we must prove that no transitive solvable subgroup of $H$ is able to leave $\mathbb Z^2$ fixed.

By Theorem 1.11 of \cite{GordonWilson:IsomGrpsOfRiemSolv}, we see that any solvable subgroup of $H$ which acts transitively must be conjugate to $B\ltimes \mathbb R^2 \times S$.  However, none of these is able to leave $\mathbb Z^2$ fixed and so the homogeneous quotient $G/\mathbb Z^2$ is not a solvmanifold.

\subsection{Testing for  solvability}
In practice, one might be handed a homogeneous space and need to know if it is a solvmanifold.  We give a simple procedure which is guaranteed to detect being a solvmanifold, as long as $\widetilde{G/H}$ is isometric to a solvmanifold with linear isometry group.  We know of no general technique that is guaranteed to work without the aforementioned condition.

\begin{prop}\label{prop: test for isometric to solv}  Let $G$ be a  Lie group with Levi decomposition $G_1G_2$ and $H$ a closed, Ad-compact, connected subgroup.  Denote the Iwasawa subalgebra of $\mathfrak g_1$ by $i(\mathfrak g_1)$.  If the following equation holds
	\begin{equation}\label{eqn: test for isometric to solv}
	i(\mathfrak g_1) + \mathfrak g_2 + \mathfrak h = \mathfrak g
	\end{equation}
then $G/H$ is a solvmanifold (relative to any $G$-invariant metric).  Further, consider $G/H$ with a $G$-invariant metric.  If $G/H$ is  locally isometric to a solvmanifold with linear isometry group, then Eqn.~\ref{eqn: test for isometric to solv} holds.
\end{prop}

\begin{proof}
If Eqn.~\ref{eqn: test for isometric to solv} holds, then the solvable subgroup of $G$ with Lie algebra $i(\mathfrak g_1) + \mathfrak g_2$ has an open orbit.  By Lemma \ref{lemma: open orbit implies transitive action}, said group acts transitively and $G/H$ is a solvmanifold.

Now suppose that $G/H$ is a simply-connected solvmanifold with linear isometry group.  While $G$ might not be a subgroup of $Isom(G/H)$ (as the action might only be almost effective), some quotient of $G$ is a subgroup.  We replace $G$ by this possibly non-simply-connected quotient.

Let $G_1$ be a Levi factor of $G$ which is a subgroup of the linear group $Isom(G/H)$.  Decompose $G_1 = G_{nc}G_c$ as a product of semi-simple ideals of non-compact and compact, respectively.  Since $G_1$ is linear, $G_{nc}$ has an Iwasawa decomposition $KAN$ where $K$ is compact.  By the arguments given in the proof of Proposition \ref{prop: linear isom implie strongly solvable}, we see that $KG_c$ stabilizes some point $p\in G/H$.  Now the group $ANG_2$ acts transitively which implies the orbit at $eH$ is open which implies Eqn.~\ref{eqn: test for isometric to solv} holds.

\end{proof}

\section{Proofs of main theorems}  
\label{sec:proof of main theorems}

We begin by proving the the first half of Theorem \ref{thm: admissible almost compl solv has linear isometry group and reconstructable from metric lie algebra data}, which states the following.

\begin{thm}\label{thm: linear isom grp}
Let $S$ be a simply-connected, almost completely solvable Lie group with left-invariant metric $g$.  Denote the isometry group by $Isom(S,g)$.  If $\mathfrak{z(s)} \subset [\mathfrak s,\mathfrak s]$, then $Isom(S,g)$ is linear.  
\end{thm}

Using Proposition \ref{prop: linear isom implie strongly solvable}, we immediately have the following.

\begin{cor}\label{cor: almost compl solv admissible implies strongly solvable} Let $S$ be a simply-connected, almost completely solvable Lie group such that $\mathfrak{z(s)} \subset [\mathfrak s,\mathfrak s]$.  For any left-invariant metric $g$, $(S,g)$ is strongly solvable.
\end{cor}

Our strategy for the proof is as follows.  Denote our solvmanifold by $M$.  Let $S$ be a simply-connected solvable group of isometries acting simply transitively on $M$.  Then 
	$$Isom(M)= S\ (Isom(S,g)_p)$$ 
for any $p\in M$.  Observe that $Isom(S,g)_p$ being compact implies that our group $Isom(S,g)$ has finitely many connected components and so to prove linearity, it suffices to do so for the connected component of $Isom(S,g)$, see  Section \ref{sec:prereq linear}.

Denote the connected isometry group by $G$ and let  $G =G_1 G_2$ be a Levi decomposition.  By the result of Goto above (see Section \ref{sec:prereq linear}), to prove $G$ is linear, it suffices to prove  $G_1$ is linear and $Rad(G,G)$ is simply-connected.  By Lemma \ref{lemma: Rad(G,G) is s.c.}, $Rad(G,G)$ is simply-connected.

\begin{lemma}\label{lemma: levi factor of isom is linearizable} The maximal compact subalgebra of $Lie~G_1$ is tangent to a compact subgroup of $G_1$.  Furthermore, $G_1$ is linearizable.
\end{lemma}

We prove this lemma below after establishing some technical results.  Note, in general, compactness does not guarantee linearity - this extra property is a consequence of being isometries of a special solvmanifold.

\subsection{A most compatible choice of $G_p$ in $G$}
To prove Lemma \ref{lemma: levi factor of isom is linearizable}, we will show that the maximal compact subalgebra of $\mathfrak g_1 = Lie~G_1$ is a subalgebra of $\mathfrak g_p$, for some $p\in M$, where $\mathfrak g_p = Lie~G_p$.  Consequently, the maximal compact subalgebra of $\mathfrak g_1$ will be tangent to a closed subgroup of a compact group, i.e. tangent to a compact group.

At this point, we fix a choice of algebra $\mathfrak s''$ in standard position so that $\mathfrak s'' \supset i(\mathfrak g_{nc})$ for some given Iwasawa subalgebra of $\mathfrak g_{nc}$.  See Section \ref{sec:prereq isom grp solvmfld} for more discussion.

By Proposition \ref{prop: existence of almost faithful representation},  there exists an almost faithful representation
	$$\rho: G \to GL(V)$$
Observe, for $p\in M$ with stabilizer subgroup $G_p$, $\rho(G_p)$ is compact and so $\rho (\mathfrak g_p)$ is fully reducible.  Moreover, the following is true.
\begin{lemma}
Let $\mathfrak g$ be the isometry algebra of a solvmanifold and $\rho$ a faithful representation.
	\begin{enumerate}  
		\item If $\mathfrak g = \mathfrak g_1 \ltimes \mathfrak g_2$ is a Levi decomposition of $\mathfrak g$, then $\rho(\mathfrak g_1) \ltimes \rho(\mathfrak g_2)$ is a Levi decomposition of $\rho(\mathfrak g)$.
		\item Further, we may change the Levi factor $\mathfrak g_1$ to assume $\rho(\mathfrak g_1)$ commutes with $mfr(\rho(\mathfrak g_2))$ without changing our choice of $\mathfrak f$ or $\mathfrak s''$.
	\end{enumerate}
\end{lemma}

\begin{proof}
The first claim is true in general for any Lie algebra and representation.  We prove the second claim.

   By Lemma \ref{lemma: levi factor commutes with mfr(g2)}, we know there exists a Levi factor of $\rho(\mathfrak g)$ so that $\rho(\mathfrak g_1)$ commutes with $mfr(\rho(\mathfrak g_2))$.  However, the Levi factors of $\mathfrak g$ are all conjugate via the nilradical of $G_2$ (Mostow, see Theorem \ref{thm: mostow - mfr are conjugate}), which is a subgroup of $F$ (Lemma \ref{lemma: nilrad(G) < S}).  Thus there exists $f\in F$ so that $\rho(f) \rho(\mathfrak g_1) \rho(f)^{-1}$ commutes with $mfr(\rho(\mathfrak g_2))$.
   
As $\mathfrak f$ is self-normalizing, we see that $Ad(f)(\mathfrak g_1\cap \mathfrak f) = Ad(f)\mathfrak g_1 \cap \mathfrak f$ and decomposing this new Levi factor will give a different Iwasawa decomposition, but the same $\mathfrak f$ (see Section \ref{sec:prereq isom grp solvmfld}).  Moreover, as $F$ normalizes $S''$, $\mathfrak s''$ does not change.
\end{proof}

By Lemma \ref{lemma: levi factor commutes with mfr(g2)}, we see that
	$$mfr(\rho(\mathfrak g)) = \rho(\mathfrak g_1) + mfr(\rho(\mathfrak g_2))$$
after having fixed a choice of $mfr(\rho(\mathfrak g_2))$.  This is a Lie algebra direct sum.

\begin{lemma} There exists $p\in M$ such that $L=G_p$ satisfies
	$$\rho(N_\mathfrak l (\mathfrak s'') )< \rho(\mathfrak g_1) + mfr(\rho(\mathfrak g_2))$$
	\end{lemma}
\begin{proof}
	Recall, $N_L(S'')$ is the closed subgroup of $L$ which normalizes $S''$.  	As the group $N_L(S'')$ is compact, so is its image $\rho(N_L(S''))$.  Thus $\rho(N_\mathfrak l (\mathfrak s''))$ is fully-reducible and is conjugate to a subalgebra of $\rho(\mathfrak g_1) + mfr(\rho(\mathfrak g_2)$.  To finish, we will carefully use the fact that a conjugate of $\rho(N_\mathfrak l (\mathfrak s''))$  satisfies the above inclusion by a result of Mostow (Theorem \ref{thm: mostow - mfr are conjugate}).
	
	Recall, $G=G_1G_2=G_1F$ and so $\rho(N_\mathfrak l (\mathfrak s''))$ is conjugate to a subalgebra of $\rho(\mathfrak g_1) + mfr(\rho(\mathfrak g_2)$ by an element of $F$.  As $F$ normalizes $S''$, we see that there is an $f\in F$ such that
		$$ N_{Ad(f)\mathfrak l}(\mathfrak s'') < \rho(\mathfrak g_1) + mfr(\rho(\mathfrak g_2)$$
Observing that $fLf^{-1} = f G_pf^{-1}=  G_{f\cdot p}$, our proof is complete.
\end{proof}

As $\mathfrak f = N_\mathfrak l (\mathfrak s'')\ltimes \mathfrak s''$ and $\mathfrak g_2< \mathfrak f$, we see that $\mathfrak f$ is stable under $ad~N_\mathfrak l(\mathfrak s'')$ and $ad~\mathfrak g_2$.  Thus we have
	$$\rho(N_\mathfrak l (\mathfrak s'')) < \rho( N_{\mathfrak g_1}(\mathfrak f)) + mfr(\rho(\mathfrak g_2))$$
However, $\mathfrak f$ is self-normalizing, which yields $N_{\mathfrak g_1}(\mathfrak f) = \mathfrak g_1 \cap \mathfrak f = \mathfrak m + \mathfrak a_1 + \mathfrak n_1 + \mathfrak g_c$.  This gives
	$$ \rho( N_\mathfrak l(\mathfrak s'') ) < \rho(\mathfrak m + \mathfrak a_1 + \mathfrak n_1 + \mathfrak g_c) + mfr( \rho (\mathfrak g_2)) = (\rho(\mathfrak g_c) + mfr(\rho(\mathfrak g_2))) + \rho(\mathfrak m \ltimes (\mathfrak a_1 + \mathfrak n_1))$$
Observe that the first and second sums of the right hand side are Lie algebra direct sums.

\begin{defin} Let $H$ be a Lie group with finitely many connected components.  We denote a maximal compact subgroup of $H$ by $MC(H)$ and likewise we denote the Lie algebra of such a subgroup by $mc(\mathfrak h)$.  
\end{defin}

\begin{remark} There is little ambiguity in our choice of notation as maximal compact subgroups are all conjugate.  In our definition of $mc(\mathfrak h)$, it is important to note that we have fixed a choice of Lie group $H$ with Lie algebra $\mathfrak h$.  Moreover, we observe that for a linear Lie group $H$, $mc(\mathfrak h) < mfr(\mathfrak h)$.
\end{remark}

\begin{lemma}\label{lemma: description of N_l(s)}  We may choose a point $p\in M$ so that $L=G_p$ satisfies
	$$ \rho(N_\mathfrak l (\mathfrak s'')) = \rho(\mathfrak g_c) + mc( \rho (\mathfrak g_2)) + \rho(\mathfrak m)$$	
\end{lemma}

\begin{proof}
	We begin by showing that the left-hand side is contained in the right-hand side.  First we show that we may assume $\rho(N_\mathfrak l (\mathfrak s'')) < \rho(\mathfrak g_c) + mfr( \rho (\mathfrak g_2)) + \rho(\mathfrak m + \mathfrak a_1) $.

By the work above, we know that $\rho( N_\mathfrak l(\mathfrak s'') ) <  (\rho(\mathfrak g_c) + mfr(\rho(\mathfrak g_2))) + \rho(\mathfrak m \ltimes (\mathfrak a_1 + \mathfrak n_1))$.  Observe that the maximal fully reducible subalgebra of 	righthand side is $(\rho(\mathfrak g_c) + mfr(\rho(\mathfrak g_2))) + \rho(\mathfrak m +   \mathfrak a_1  )$ and the nilradical is (modulo the center) $\mathfrak n_1$. Thus, we can conjugate $\rho(N_\mathfrak l(\mathfrak s''))$ into the maximal fully reducible $(\rho(\mathfrak g_c) + mfr(\rho(\mathfrak g_2))) + \rho(\mathfrak m +   \mathfrak a_1  )$ via $N_1<G_1$.  It is important to note that, as before, this conjugation does not change $S$ and only changes the base point  of $L=G_p$.

As the group $N_L(S'')$ is compact, the algebra $\rho(N_\mathfrak l(\mathfrak s'')$ is contained in some maximal compact subalgebra of $(\rho(\mathfrak g_c) + mfr(\rho(\mathfrak g_2))) + \rho(\mathfrak m +   \mathfrak a_1  )$.  Observe that the group $\rho(G_c) \ mfr(\rho(G_2)) \ \rho(M) \ \rho(A_1)$ has a single maximal compact subgroup 
	$$\rho(G_c) \ MC(\rho(G_2)) \ \rho(M)  $$
since these four factors all commute with each other, $MC(\rho(G_2)) < mfr(\rho(G_2))$, $\rho(M)$ is compact, and $\rho(A_1)$ is fully non-compact (see Lemma \ref{lemma: rho M and A}).  This shows half of our lemma, namely, $\rho(N_\mathfrak l (\mathfrak s'')) < \rho(\mathfrak g_c) + mc( \rho (\mathfrak g_2)) + \rho(\mathfrak m)$.
	
We now work to show the reverse inclusion.  The above inclusion implies
	$$\rho( \mathfrak f ) =   \rho( N_\mathfrak l (\mathfrak s'')  +\mathfrak s'' )   < \rho(\mathfrak g_c) + mc( \rho(\mathfrak g_2)) + \rho(\mathfrak m) + \rho(\mathfrak s'') < \rho(\mathfrak f)    $$
and so the inclusions above are equalites.  For more details, see Section \ref{sec:prereq isom grp solvmfld}.

Next we show $( \rho(\mathfrak g_c) + mc( \rho(\mathfrak g_2)) + \rho(\mathfrak m) ) \cap \rho(\mathfrak s'') = \{0\}$.  To see this, we investigate the eigenvalues of $ad~X$ for any $X$ in the intersection of interest.  Such $ad~X$ must preserve $\mathfrak s''$ as $\mathfrak s''$ is a subalgebra.  
However, $X\in \rho(\mathfrak g_c) + mc( \rho(\mathfrak g_2)) + \rho(\mathfrak m) $ being tangent to a compact subgroup implies that $ad~X$ has purely imaginary eigenvalues and is fully reducible.  By our hypothesis  that $\mathfrak s$ is almost completely solvable, and Theorem \ref{thm: std. mod. preserves almost compl. solv.}, such $ad~X$ can have only zero eigenvalues.  Now $ad~X$ is fully reducible with zero eigenvalues, i.e., $X\in \mathfrak z(\rho(\mathfrak s))$.  Finally, our hypothesis that $\mathfrak z(\mathfrak s) < [\mathfrak s,\mathfrak s]$, together with Theorem \ref{thm: admissible almost compl solv impies center in [r,r]} and Corollary \ref{cor:reductive elements of commutator of solvable}, forces $X=0$.

Similarly, $N_\mathfrak l (\mathfrak s'') \cap \mathfrak s'' = \{0\}$ and we see that
	$$\dim  \rho(N_\mathfrak l (\mathfrak s'')) = \dim \rho(\mathfrak g_c) + mc( \rho (\mathfrak g_2)) + \rho(\mathfrak m)$$
As one of these spaces is a subspace of the other, we now see that $\rho(N_\mathfrak l (\mathfrak s'')) = \rho(\mathfrak g_c) + mc( \rho (\mathfrak g_2)) + \rho(\mathfrak m)$.
	
\end{proof}

\begin{remark*} In the sequel, we denote by $K(G_{nc})$ the closed subgroup of $G_{nc}$ whose Lie algebra $\mathfrak{k(g_{nc})}$ is the compactly embedded subalgebra of an Iwasawa decomposition $\mathfrak{g_{nc}} = \mathfrak{k(g_{nc})} + \mathfrak a_1 + \mathfrak n_1$.
\end{remark*}

\begin{prop}\label{prop: maximal compact of non-compact semi-simple is in stabilizer}
There exists $p\in M$ such that $\rho(G_p) = \rho( K(G_{nc})) \rho(N_L(S''))$.  Consequently, $Lie ~ K(G_{nc}) < \mathfrak g_p$,  $K(G_{nc})_0 < G_p$, and $K(G_{nc})$ is compact.
\end{prop}

\begin{remark} A priori, the group $L=G_q$ might be the stabilizer of a point $q$  different from $p$.
\end{remark}

\begin{proof}
The first step is to show $\rho(G_p) < \rho( K(G_{nc})) \rho(N_L(S''))$.  As $G_p$ is compact, for any $p\in M$, and maximal fully reducible subalgebras are all conjugate, we see that $\rho(G_p)_0 < \rho(G_1) \ mfr(\rho(G_2))$ for some $p\in M$.  Furthermore, we may assume (by conjugation) that 
	$$\rho(G_p)_0 < \rho(K(G_{nc})G_c) \ MC(\rho(G_2))$$
as maximal compact subgroups are conjugate.  Note, although the subgroup $K(G_{nc})$ of $G_1$ might not be compact, the image under $\rho$ must be compact as $\rho(G_1)$ is a linear semi-simple Lie group.

By the Lemma above, we then have $\rho(G_p)_0 <  \rho(K(G_{nc})G_c) \ MC(\rho(G_2)) \ \rho(N_L(S'')) = \rho(K(G_{nc}) \ N_L(S'') )$.  We will show this containment is an equality by counting dimensions.  Recall, $\rho$ being an almost faithful representation yields

\begin{eqnarray}
	\dim \rho(K(G_{nc}) \ N_L(S'') ) &=& \dim K(G_{nc}) \ N_L(S'') \\
													&=& \dim K(G_{nc}) - \dim M + \dim N_L(S'')\\
													&=& \dim K(G_{nc}) - \dim M + \dim F - \dim S''\\
													&=& \dim K(G_{nc})F -\dim S''\\
													&=& \dim G - \dim S''\\
													&=& \dim G_p\\
													&=& \dim \rho(G_p)
\end{eqnarray}
The second equality follows from $M =   K(G_{nc}) \cap N_L(S'')$, by Lemma \ref{lemma: description of N_l(s)}.  The third equality follows from the fact that $F = N_L(S'')\ltimes S''$ is the normalizer of $S''$ and that $S''$ acting simply transitively implies $S''\cap N_L(S'') \subset S''\cap L = \{e\}$.  The fourth equality follows from the definition of $F = MANG_cG_2$, which implies $K(G_{nc})\cap F = M$.  The fifth equality follows from the definition of $F$.  The sixth equality follows from the fact that $S$ acts simply transitively, and so $S''\cap G_p = \{e\}$.  The last equality follows from $\rho$ being almost faithful.

\end{proof}

From this proof, we immediately see the following corollary.

\begin{cor} \label{cor: stabilizer respects levi decomposition} Let $p\in M$ be as above.  Then $(G_p)_0 = K(G_{nc}) ~ G_c ~ (G_p \cap G_2) = (G_p\cap G_1) (G_p\cap G_2)$.
\end{cor}

The proof amounts to checking this equality at the Lie algebra level, which follows using the faithful representation $\rho$.

\begin{defin}  We say our Levi decomposition is compatible if $(G_p)_0 = K(G_{nc}) ~ G_c ~ (G_p \cap G_2) = (G_p\cap G_1) (G_p\cap G_2)$, for $p = e\in S''$.  Such a decomposition exists for admissible, almost completely solvable groups $S$ by the corollary above.
\end{defin}

\subsection{Linearity of the Levi factor $G_1$}  To show $G_1$ is linear, we give an explicit linearization of $G_1$.  This will be used in the following section when describing how to construct the isometry group from algebraic data of $\mathfrak s$.

A key ingredient to the proof is that the group $K(G_{nc})$ is compact.  However, we remind the reader that compactness of $K(G_{nc})$ alone is not enough to guarantee $G_{nc}$ is linear as there do exist semi-simple groups which finitely cover linear groups, which cannot be linear themselves, so care must be taken.

\begin{lemma}\label{lemma: s2 properties} Let $S$ be an almost completely solvable Lie group such that $\mathfrak z(\mathfrak s) \subset [\mathfrak s,\mathfrak s ]$.  Assume $S''$ is simply-connected and in standard position.  Let $G_1G_2$ be a compatible decomposition of the connected isometry group of $S$ relative to $G_p$.  Then 
\begin{enumerate}
	\item $S'' = S_1 S_2$ where $S_1 = S'' \cap G_1$ and $S_2 = S''\cap G_2$.
	\item  $S_2$ is a normal subgroup of $G$.
	\item $G_2 = S_2 (G_2)_p$ for any $p\in M$.
	\item Relative to the Killing form of $\mathfrak g_2$, $\mathfrak s_2$ is the orthogonal complement to $\mathfrak g_2\cap \mathfrak g_p$ in $\mathfrak g_2$ for any $p\in S$.
\end{enumerate}
\end{lemma}

\begin{proof}
Let $p\in S''$ be as above, satisfying the definition of compatible Levi decomposition.  We begin by defining $\mathfrak s_2$ to be the orthogonal complement of $\mathfrak g_2 \cap \mathfrak g_p$ in $\mathfrak g_2$ relative to the Killing form $B_{\mathfrak g_2}$ of $\mathfrak g_2$.  We will show that $\mathfrak s_2 = \mathfrak s'' \cap \mathfrak g_2$.

First, we show that $\mathfrak s_2\subset \mathfrak s''$ by showing it is  orthogonal to $\mathfrak f \cap \mathfrak g_p$ relative to the Killing form $B_\mathfrak f$ of $\mathfrak f$.    By our choice of $p\in S''$, we have $\mathfrak f \cap \mathfrak g_p =  N_\mathfrak l (\mathfrak s'') = \mathfrak m + \mathfrak g_c + (\mathfrak g_2 \cap \mathfrak g_p)$.  Given $X\in \mathfrak f \subset \mathfrak g$ and $Y\in \mathfrak g_2$, we see that
	$$B_\mathfrak f (X,Y) = tr_\mathfrak f \ ad~X\circ ad~Y  =tr_{\mathfrak g_2} \ ad~X\circ ad~Y = B_\mathfrak g(X,Y)     $$
where $B_\mathfrak g$ is the Killing form of $\mathfrak g$.  As $\mathfrak g_1$ and $\mathfrak g_2$ are orthogonal relative to $B_\mathfrak g$, we only need to verify that $\mathfrak s_2$ is orthogonal to $\mathfrak g_2\cap \mathfrak g_p$ relative to $B_\mathfrak f$.  However, the above equation shows for $X\in \mathfrak g_2 \subset \mathfrak f$ and $Y\in \mathfrak g_2$, $B_\mathfrak f (X,Y) = B_{\mathfrak g_2}(X,Y)$, and so $\mathfrak s_2$ is orthogonal to $\mathfrak f \cap \mathfrak g_p$ relative to $B_\mathfrak f$.

Define $\mathfrak s_1 = \mathfrak s'' \cap \mathfrak g_1$.  We now have $\mathfrak s_1 + \mathfrak s_2 \subset \mathfrak s''$.  It immediately follows from the definition of a compatible decomposition that  $\mathfrak g_1 = \mathfrak s_1 + (\mathfrak g_1 \cap \mathfrak g_p)$ and   we have the following inclusion,
	$$\mathfrak g = \mathfrak g_1 + \mathfrak g_2 = \mathfrak g_1 + \mathfrak s_2 + (\mathfrak g_2\cap \mathfrak g_p) \subset \mathfrak s'' + \mathfrak g_p = \mathfrak g$$
As $\mathfrak s'' \cap \mathfrak g_p = \{0\}$,  counting dimensions shows that   $\mathfrak s''= \mathfrak s_1 + \mathfrak s_2$, which shows the first claim.

Next we prove the second claim.  Recall, the nilradical $nilrad(\mathfrak g_2)$ of $\mathfrak g_2$ is contained in the kernel of the Killing form $B_{\mathfrak g_2}$.   As $\mathfrak s_2$ is the orthogonal complement of $\mathfrak g_2\cap \mathfrak g_p$ in $\mathfrak g_2$, we see that 
	$$nilrad(\mathfrak g_2) \subset \mathfrak s_2$$
To finish, we apply the fact that any derivation of a solvable Lie algebra takes its image in the nilradical (Lemma \ref{lemma: derivation of solvabe goes to nilradical}).   Take $X\in \mathfrak g$.  As $\mathfrak g_2$ is an ideal of $\mathfrak g$, $ad~X|_{\mathfrak g_2} \in Der(\mathfrak g_2)$ and so $ad~X: \mathfrak g_2 \to nilrad(\mathfrak g_2)$.  Consequently,
	$$ad~X : \mathfrak s_2 \to \mathfrak s_2$$
that is, $\mathfrak s_2$ is an ideal of $\mathfrak g$.  This implies the analogous statement at the group level.

The third claim follows immediately from $G_2$ and $S_2$ being normal subgroups, the Killing form being $Ad$-invariant, and the property $G_{g\cdot p} = g G_p g^{-1}$.

To prove the fourth claim, one just uses the fact that the Killing form is invariant under the automorphism group together with $\mathfrak s_2$ being ideal.

\end{proof}

We are now in a position to prove Lemma \ref{lemma: levi factor of isom is linearizable}, that is, $G_1$ is linearizable.

\begin{proof}
	Let $G =G_1G_2$ be a compatible Levi decomposition of the connected isometry group $G$, relative to $G_p$.    	
Write $G_1 = K S_1$ where $K=K(G_{nc})G_c < G_p$ is a maximal compact subgroup of $G_1$ and $S_1=G_1\cap S''$ is an Iwasawa subgroup of $G_{nc} < G_1$.
	
	From  Corollary \ref{cor: stabilizer respects levi decomposition},  Lemma \ref{lemma: s2 properties}, and  the fact that $G_2\cap G_p$ commutes with $G_1$,  we may write our solvmanifold $M$ as follows
	$$ M= G /G_p =  G_2/(G_2\cap G_p) \times G_1/K = S_2 \times G_1/K$$
with the action of $g\in G_1$ on $M$ given by
		$$ g\cdot (r, lK) = (grg^{-1}, glK)$$
for $r\in S_2$ and $l\in G_1$.  
		
Observe that the finite center $Z$ of $G_1$ is contained in the maximal compact $K$ as the conjugacy of maximal compact subgroups implies this finite central group is in every maximal compact subgroup.  Consequently, $G_1/K = Ad(G_1)/Ad(K)$, where $Ad(G_1)$ denotes the adjoint group $G_1/Z$, and we see that the action of $G_1$ on $M$ gives a homomorphism into the linear group
		$$G_1 \to Aut(S_2) \times Ad(G_1)$$	
As the group $G_1$ acts effectively on $M$, our homomorphism is an isomorphism onto its image, and we have linearized $G_1$.
\end{proof}
This completes the proof of Theorem \ref{thm: linear isom grp}.

\subsection{Non-linear isometry groups}
It is natural to ask if we are able to relax any of the hyptheses of the theorem, in general.  Among completely solvable Lie groups, Example \ref{ex: bad metric on H^2 x R} shows that we do need the condition of admissibility to guarantee linearity of the isometry group.  The following example shows that it is hard to relax the condition of being almost completely solvable.

\begin{example}  There exists a Riemannian solvable Lie group $S$ such that $\mathfrak{z(s)} = 0$ and $Isom(S)$ is not linear.
\end{example}

To construct our example, we will start with a simple Riemannian solvable Lie group and perform a normal modification.  That is, we will find a different transitive solvable group which acts transitively and has the desired properties, cf. Section \ref{sec:prereq isom grp solvmfld}.

Start with $G=G_1\times G_2$ where $G_1$ is the group given in Example \ref{ex: bad metric on H^2 x R} and $G_2$ is the five dimensional nilpotent Lie group whose Lie algebra is defined as follows.  We let $\{e_1, \dots, e_5\}$ be a basis of $\mathfrak g_2$ and consider the relations
	$$ [e_1,e_2] = e_4 \quad [e_2,e_3]=e_5$$
These define a two-step nilpotent Lie structure by including the obvious non-zero brackets from anti-symmetry and all other relations among these basis vectors being zero.

Declaring $\mathfrak g_1$ to be orthogonal to $\mathfrak g_2$ and using the above basis as an orthonormal basis, we endow $G_2$ with a left-invariant metric and have
	$$Isom(G) \supset Isom(G_1)\times Isom(G_2)$$
As $Isom(G_1)$ is not linear, the same is true for $Isom(G)$.  Next we modify $\mathfrak g$ by twisting on a skew-symmetric derivation.

Consider the derivation $D\in Der(\mathfrak g)$ which is defined to be zero on $\mathfrak g_1$ and on $\mathfrak g_2$ defined via
	\begin{eqnarray*}
		D(e_1) &=& -e_3\\
		D(e_2) &=& 0\\
		D(e_3) &=& e_1\\
		D(e_4) &=& e_5\\
		D(e_5) &=& -e_4
	\end{eqnarray*}
Clearly this is skew-symmetric relative to the given inner product on $\mathfrak g$.  Note that $D$ is non-singular on the center of $\mathfrak g_2$.

Now consider the solvable Lie group $S$ with Lie algebra given by $H^2 \times \mathbb R \ltimes \mathfrak g_2$ where $H^2$ is the Iwasawa subalgebra of $\mathfrak{sl}(2,\mathbb R)$ and $\mathbb R$ acts on $\mathfrak g_2$ via $D$.  As $S$ acts simply-transitively, by isometries, on $G$, we see that $S$ has all the desired properties for our example.

\section{Algorithm for full isometry group}
\label{sec: Algorithm for full isometry group}
In this section, we describe how one can recover the full isometry group of any left-invariant metric on an almost completely solvable Lie group, as long as the group is admissible.  The proves the second half of Theorem \ref{thm: admissible almost compl solv has linear isometry group and reconstructable from metric lie algebra data}.

\begin{defin}  Let $\mathfrak s$ be a solvable Lie algebra.  We say $\mathfrak s = \mathfrak s_1 + \mathfrak s_2$ is an $LR$-decomposition if
	\begin{enumerate}
	\item $\mathfrak s_1$ is trivial or  
	an Iwasawa algebra of a non-compact semi-simple Lie algebra $\mathfrak g_1$,	
		\item $\mathfrak s_2$ is a (possibly trivial) ideal,
	\item The restriction of the adjoint representation $ad|_{\mathfrak s_2}: \mathfrak s_1 \to Der(\mathfrak s_2)$ extends to a representation $\rho : \mathfrak g_1 \to Der(\mathfrak s_2)$.
	\end{enumerate}
\end{defin}
Given any maximal compact subalgebra $\mathfrak k(\mathfrak g_1) \subset \mathfrak g_1$, we may  identify $\mathfrak s_2 + \mathfrak s_1 \simeq \mathfrak s_2 + \mathfrak g_1/ \mathfrak k(\mathfrak g_1)$ so that the right-hand side inherits the inner product from $\mathfrak s$.  Observe that $\mathfrak k(\mathfrak g_1)$ acts on $\mathfrak s_2 + \mathfrak g_1/ \mathfrak k(\mathfrak g_1)$ via $\rho \times ad$.
	
Let $\mathfrak s$ be a metric solvable Lie algebra.  We say an LR-decomposition is \emph{suitable} if there exists a maximal compact subalgebra $\mathfrak k(\mathfrak g_1) \subset \mathfrak g_1$ such that $(\rho \times ad) (\mathfrak k(\mathfrak g_1)) \subset \mathfrak{so}(\mathfrak s_2 + \mathfrak g_1/ \mathfrak k(\mathfrak g_1))$.  We say a suitable LR-decomposition is \emph{maximal} if $\mathfrak s_1$ has maximal dimension.

\begin{remark} The semi-simple algebra $\mathfrak g_1$ above is uniquely determined by its Iwasawa algebra $\mathfrak s_1$ (see \cite{Araki:OnRootSystemsAndAnInfinitesimalClassificationOfIrreducibleSymmetricSpaces}) and, as such, we use the notation $\mathfrak s_1 = i(\mathfrak g_1)$.  The idea of an LR-decomposition was first considered by Alekseevskii-Cortes where it is called an SR-decomposition.  We have changed the terminology to reflect the historical notation used for a Levi decomposition, from which this comes.  

The term suitable in \cite{AlekseevskyCortes:IsomGrpsOfHomogQuatKahlerMnflds} is used slightly differently, but to the same end.  Our definition is more technical due to our not requiring $\mathfrak s_1$ to be orthogonal to $\mathfrak s_2$. 

\end{remark}

Let $\mathfrak s$ be a metric Lie algebra.  We denote the derivation algebra of $\mathfrak s$ by $Der(\mathfrak s)$  and the skew-symmetric derivations by $\mathfrak{d(s)}$.  This gives rise to the isometry algebra
	$$\mathfrak{d(s)} \ltimes \mathfrak s$$
that is, this is the Lie algebra of a group of isometries of the simply-connected Riemannian Lie group $S$.  When $\mathfrak s$ is completely solvable and unimodular, this is the full isometry algebra (see \cite{GordonWilson:IsomGrpsOfRiemSolv}).  However, in general, this is not the case.

Given an LR-decomposition $\mathfrak s = \mathfrak s_1 + \mathfrak s_2$, we denote by $\mathfrak{d_0(s_2)}$ the skew-symmetric derivations of $\mathfrak s$ which restrict to derivations on $\mathfrak s_2$ and vanish on $\mathfrak s_1$.

\begin{thm}\label{thm: LR algorithm} Let $S$ be an almost completely solvable, simply-connected Riemannian Lie group with metric Lie algebra $\mathfrak s$.  
Any suitable LR-decomposition $\mathfrak s = \mathfrak s_1 + \mathfrak s_2$    
gives rise to an isometry algebra ${\mathfrak g \subset Lie~Isom(S)}$ defined as follows
	$$\mathfrak g = \mathfrak g_1 + \mathfrak{d_0(s_2)} + \mathfrak s_2$$
where $\mathfrak s_2$ is an ideal, $\mathfrak{d_0(s_2)}$ acts on $\mathfrak s_2$ as derivations, 
$\mathfrak g_1$ acts on $\mathfrak s_2$ via $\rho$ (the extension of the adjoint action of $\mathfrak s_1$ on $\mathfrak s_2$), and $[\mathfrak g_1, \mathfrak{d_0(s_2)}]=0$.

Furthermore, assume $\mathfrak s$ is admissible, i.e., $\mathfrak{z(s)} \subset [\mathfrak s,\mathfrak s]$.   Then a maximal LR-decomposition of the  modification $\mathfrak s''$ in standard position is unique (up to conjugacy in $Lie~Isom(S)$) and the above isometry algebra for $\mathfrak s''$ is the full isometry algebra.
\end{thm}

\begin{remark} We remind the reader that the standard modification algorithm of Gordon-Wilson relies only on metric Lie algebra data, see Section \ref{sec:prereq isom grp solvmfld}.  Further, in the special case that $\mathfrak s$ is completely solvable, we have that $\mathfrak s = \mathfrak s''$ and so the isometry group can be computed by considering $LR$-decompositions of the original algebra.
\end{remark}

\begin{proof}  As in \cite{AlekseevskyCortes:IsomGrpsOfHomogQuatKahlerMnflds}, we  prove the first part of our theorem by finding a Lie group $G$ with $Lie~G = \mathfrak g$, and compact subgroup $K$ such that $Lie~K = \mathfrak{k(\mathfrak g_1)} +\mathfrak{d_0(s_2)}$, so that we may identify  $S\simeq G/K$.

	Consider the Lie algebra $(\mathfrak g_1+\mathfrak{d_0(s_2)})\ltimes \mathfrak s_2$.  The algebra $\mathfrak g_1$ is a subalgebra of $ad(\mathfrak g_1) \times Der(\mathfrak s_2)$, via the adjoint representation of $\mathfrak g_1$, and so we may consider the group $G_1 < Inn(\mathfrak g_1) \times Aut(\mathfrak s_2)$.  As $G_1$ is semi-simple subgroup of a linear group, it is semi-algebraic and thus a  closed, linear Lie group.
	
	Let $A_0(\mathfrak s_2)$ denote the subgroup of $Aut(S_2)$ with Lie algebra $\mathfrak{d_0(s_2)}$.  By construction, this is a closed semi-algebraic group.
	
	Consider the simply-connected group $S_2$ with Lie algebra $\mathfrak s_2$.  As $Aut(S_2) \simeq Aut(\mathfrak s_2)$, we may form the group
		$$G = G_1A_0(\mathfrak s_2) \ltimes S_2$$
This is a closed subgroup of the linear group $(Inn(\mathfrak g_1) \times Aut(\mathfrak s_2) )\ltimes S_2$ with Lie algebra $\mathfrak g$.

	Let $\mathfrak k = \mathfrak k(\mathfrak g_1) + \mathfrak{d_0(\mathfrak s_2)}$ where $\mathfrak{k(g_1)}$ is the subalgebra appearing above for our suitable LR-decomposition of $\mathfrak s$.  As $\mathfrak{k(g_1)}$  is a maximal compact algebra and the group $G_1$ is linear, the connected subgroup $K(G_1) < G_1$ with Lie algebra $\mathfrak{k(g_1)}$ is compact.  Thus $K = K(G_1)A_0(S_2)$ is compact as it is a product of compact subgroups.  
	
	The quotient space $\mathfrak g/\mathfrak k$ may be naturally identified with $\mathfrak s_2 + \mathfrak g_1/ \mathfrak{k(g_1)}$ since $[\mathfrak g_1, \mathfrak{d_0(s_2)}]=0$.  Endow $\mathfrak g/\mathfrak k$ with the inner product from $\mathfrak s_2 + \mathfrak g_1/ \mathfrak{k(g_1)} \simeq \mathfrak s_2+\mathfrak s_1$.  Using this choice of inner product, we see that  the adjoint action of $K$ consists orthogonal transformations and the inner product on $\mathfrak g/\mathfrak k$ extends to a left $G$-invariant metric on $G/K$.
	
	By construction, $\mathfrak g = \mathfrak s + \mathfrak k$, $S\cap K = \emptyset$, $G=SK$, and so $S\simeq G/K$.  Further, our choice of metric on $G/K$ precisely makes these spaces isometric.  This proves the first claim.
	
	To prove the second claim, we need the hypotheses of the theorem (almost completely solvable and admissible) to guarantee that $Isom(S)$ is linear, see Theorem \ref{thm: linear isom grp}.  Lemma \ref{lemma: s2 properties} shows that the isometry group of such a space does, in fact, arise from an LR-decomposition.  This being from a maximal LR-decomposition follows from the Levi factor being a maximal semi-simple subgroup of the isometry group.  Again, uniqueness up to conjugacy follows from the same result for Levi groups.
	
\end{proof}

\begin{thm}\label{thm: algorithm for full isom group} Let $S$ be a simply-connected, admissible, almost completely solvable group and     $S''=S_1S_2$ be a maximal LR-decomposition of the modification $S''$ in standard position with associated (connected) group of isometries $G_0$ coming from Theorem \ref{thm: LR algorithm} above.  Denote the stabilizer of $e\in S$ by $K_0$.  Then the full isometry group is given by 
	$$G = KG_0$$
where $K$ is the subgroup  $N_{Aut(\mathfrak s_2) \times Aut(\mathfrak g_1)} (\mathfrak k) \cap \mathfrak{so}(\mathfrak s_2 + \mathfrak g_1/\mathfrak k(\mathfrak g_1) )$ of $Aut(\mathfrak s_2) \times Aut(\mathfrak g_1)$ for $\mathfrak k = Lie~K_0$.
\end{thm}

\begin{remark} Observe that $K$ is computed purely from metric Lie algebra information.  As $G_0$ is also computed from the same data, we see that the full isometry group $G$ can be reconstructed solely from metric Lie algebra information.
\end{remark}

\begin{proof} Let $K$ be the isotropy of the point $e\in S$; this group contains $K_0$.  By Corollary 5.2 of \cite{Mostow:FullyReducibleSubgrpsOfAlgGrps}, we may assume that $G_1$ is stable under conjugation by $K$.  

To prove our claim, it suffices to prove $S_2$ is invariant under conjugation by $K$.  Then, as above, the $K$ action on $S$ will be the same as the $K$-action on
	$$S_2 \times G_1/K\cap G_1$$
where $K$ acts on $S_2$ and $G_1$ by conjugation.  As with symmetric spaces, the isometric action $G_1/K\cap G_1$ comes from automorphisms of $\mathfrak g_1$ which normalize $\mathfrak k \cap \mathfrak g_1$.

By Lemma \ref{lemma: s2 properties}, $\mathfrak s_2$ is the orthogonal complement of $\mathfrak g_2 \cap \mathfrak k$ relative to the Killing form $B_{\mathfrak g_2}$.  As $\mathfrak g_2 \cap \mathfrak k$ is stable under $Ad_K$ and the Killing form is invariant under automorphisms, the ideal $\mathfrak s_2$ is invariant under $Ad_K$.  This implies $S_2$ is stable under conjugation by $K$.

\end{proof}

\section{Applications to homogeneous Ricci solitons}
\label{sec: homogeneous Ricci solitons}

In this section, we apply the above work to the problem of classifying homogeneous Ricci solitons by showing that any homogeneous Ricci soliton which is locally isometric to a solvmanifold is, in fact, a solvsoliton (i.e. an algebraic Ricci soliton on a solvable Lie group), see Theorem \ref{thm: strong gen alek conj reduced to s.c. case}.

Let $S$ be a simply-connected completely solvable Lie group with Ricci soliton metric.  By the work \cite{Jablo:HomogeneousRicciSolitons}, we know that $S$ is a solvsoliton.   Furthermore, it was shown there that if a homogeneous Ricci soliton admits a transitive solvable group of isometries, then it must be a simply-connected solvable Lie group with left-invariant metric.  Our strategy is to prove the isometry group of a solvsoliton is linear, thus making it a strongly solvable space (Proposition \ref{prop: linear isom implie strongly solvable}).  Then the only homogeneous spaces covered by the solvsoliton must be solvmanifolds themselves (Proposition \ref{prop: homog space covered by strong solv is solv}).

\subsection{Decomposing the group}
Recall, every solvsoliton $S$ is isometric to a completely solvable Lie group \cite{Lauret:SolSolitons}. Thus, if the group in standard position were admissible,  we would be done.  From here on, we assume that our solvsoliton $S$ is completely solvable.

If $\mathfrak{z(s)} \not \subset [\mathfrak s,\mathfrak s]$, then we may decompose $\mathfrak s = \mathfrak t + \mathfrak u$ as a sum of subalgebras where $\mathfrak u \subset \mathfrak{z(s)}$,  $\mathfrak{z(t)} = \mathfrak{z(s)} \cap \mathfrak t \subset [\mathfrak t, \mathfrak t] = [\mathfrak s,\mathfrak s]$, and    $\mathfrak t \cap \mathfrak u = \emptyset$.  Observe that both $\mathfrak{t}$ and $\mathfrak{u}$ are ideals and $\mathfrak t$ is admissible.  By Theorem 3.4 of \cite{Jablo:ConceringExistenceOfEinstein}, both $\mathfrak t$ and $\mathfrak u$ admit solvsoliton metrics and, in fact, the solvsoliton metric on $\mathfrak s$ is such that 
	$$\mathfrak s = \mathfrak t \oplus \mathfrak u$$
is an orthogonal direct sum.   (We take this opportunity to remark that Theorem 3.4, loc. cit., as stated there, is incorrect.  It should read: $\mathfrak s$ is a non-flat solvsoliton if and only if both $\mathfrak t$, $\mathfrak u$ are solvsolitons and at least one is non-flat.)

Denote by $T$, $U$ the subgroups of $S$ with Lie algebras $\mathfrak t$, $\mathfrak u$, respectively.  The above implies $S = T \times U$ as a Riemannian product.

\begin{prop}\label{prop: isom of solvsoliton splits} The isometry group of $S$ endowed with the solvsoliton metric is
	$$Isom(S) = Isom(T) \times Isom(U)$$
Consequently, $Isom(S)$ is linear.
\end{prop}

To prove this result, we need several technical lemmas.

\subsection*{Step 1.}  First we will show that the center of $\mathfrak s$ is stable under the adjoint action of $G$.

\begin{lemma} Let $G_{nc}$ denote a connected semi-simple Lie group of non-compact type.  Let $\rho: G_{nc} \to GL(V)$ be a finite dimensional representation.  If for some $p\in V$, $\rho(G_{nc})\cdot p$ is compact, then $\rho(G_{nc})\cdot p = p$.
\end{lemma}

\begin{proof}  To prove this lemma, we employ  tools from real Geometric Invariant Theory.  Our primary reference will be \cite{RichSlow}, see also \cite{EberleinJablo}.  We abuse notation and represent the induced Lie algebra representation by $\rho$.

As $\rho(G_{nc})$ is semi-simple, we may endow $V$ with an inner product $\ip{\cdot , \cdot}$ such that $\rho(\mathfrak{g}_{nc}) = \mathfrak k + \mathfrak p$, where $\mathfrak k = \{ X\in \rho(\mathfrak g_{nc}) \ | \ X^t = -X \}$ and $\{ X\in \rho(\mathfrak g_{nc}) \ | \ X^t = X \}$.

We say a point $p\in V$ is minimal if $|| g\cdot p || \geq ||p||$ for all $g\in \rho(G_{nc})$.  By Theorem 4.4 of \cite{RichSlow}, our compact orbit $\rho(G_{nc})\cdot p$ is closed and so contains a minimal point; assume $p$ is said minimal point.  Further, compactness of this orbit forces $\mathfrak p \subset \rho(\mathfrak g_{nc})_p$, see Lemmas 3.1 \& 3.2, loc. cit.

Finally, since $\rho(\mathfrak g_{nc})$ is also a semi-simple Lie algebra of non-compact type, we see that $\mathfrak k = [\mathfrak p,\mathfrak p]$ and hence $\rho(\mathfrak g_{nc})$ annihilates $p$; the connectedness of $G_{nc}$ now yields   $\rho(G_{nc})\cdot p = p$.
\end{proof}

In the following lemma, we have no algebraic restrictions on the solvable group in standard position.

\begin{lemma} Let $S$ be a solvable Lie group (with left-invariant metric) in standard position within its isometry group $G$.  The central subalgebra $\mathfrak{z(s)}$ is stable under $Ad_G$.
\end{lemma}

\begin{proof}
	Recall, we may write $G = G_{nc}F$ where $G_{nc}$ is the semi-simple factor of non-compact type within a given Levi factor $G_1$ of $G$ and $F$ normalizes $S$, see Lemma \ref{lemma: properties of f}.  Clearly, as $F$ normalizes $S$, we have that $\mathfrak{z(s)}$ is preserved under $Ad_F$.  So it suffices to prove stability under $Ad_{G_{nc}}$.
	
	We may choose $\mathfrak g_1$ such that $\mathfrak s$ contains some Iwasawa subalgebra $i(\mathfrak g_{nc})$ of $\mathfrak g_{nc}$, see Lemma \ref{lemma: properties of f}.  Consider the Iwasawa decompositions $G_{nc} = K i(G_{nc})$ and $Ad_{G_{nc}} = Ad_K Ad_{i(G_{nc})}$.  Since $Ad_{G_{nc}}$ is linear, the subgroup $Ad_K$ is compact.  Hence, for $X\in \mathfrak{z(s)}$,  the orbit $Ad_{G_{nc}} (X) = Ad_K(X)$ is compact.  By the previous lemma, we see that $X$ is fixed by $Ad_{G_{nc}}$.  Thus, $\mathfrak{z(s)}$ is stabilized by $Ad_G$.	
\end{proof}

\subsection*{Step 2.}  Next we  show that the $G = Isom(S)$ preserves each of the factors in $S=T\times U$; i.e. $G = Isom(T)\times Isom(U)$.  To study the two submanifolds $T$ and $U$, we will consider their orbits on $G/K$, where $K$ is the isotropy at $e\in S$.  First we give a geometric characterization of the tangent spaces at the identity of these two orbits.

Recall, for $X\in \mathfrak{z(s)}$, we have the following equation for the sectional curvature
	$$ K(X,Y) = \frac{1}{4} | ad~Y^* (X)|^2   $$
for all $Y\in \mathfrak s$.  This follows from \cite[Eqn. 7.30]{Besse:EinsteinMflds}, where $ad~Y^*$ denotes the metric adjoint of $ad~Y$.

\begin{lemma}  Let $R_X$ denote the curvature operator at $X\in T_eS$.  Then
	$$ \{ X\in \mathfrak{z(s)} \ | R_X = 0 \} = \mathfrak u = \mathfrak z \ominus [\mathfrak s,\mathfrak s]$$
\end{lemma}

\begin{proof}  For $X\in\mathfrak{z(s)}$, the following are equivalent
\begin{eqnarray*}
	R_X &=& 0 \\
	K(X,Y) &=& 0 \quad \mbox{ for all } Y\in \mathfrak s\\
	ad~Y^*X &=& 0 \quad \mbox{ for all } Y\in\mathfrak s\\
	\ip{X,[Y,Z]} &=& 0 \quad \mbox{ for all } Y,Z \in\mathfrak s\\
	X &\perp& [\mathfrak s,\mathfrak s]
\end{eqnarray*}
as desired.
\end{proof}

Since $\mathfrak{z(s)}$ is $Ad_G$ stable, it is $Ad_K$ stable.  Now, the above shows that $T_e U$ is $Ad_K$ stable as $K$ acts by isometries.  Furthermore, $K$ preserves the orthogonal complement of this space, namely, $T_eT$.  We are careful to note that the subspaces which are stable are subspaces of  $T_{eK}G/K$ and not necessarily the Lie algebras as subspaces of $\mathfrak g$.  However, we do have
	$$\mathfrak u + \mathfrak k \quad \mbox{ and } \quad \mathfrak t + \mathfrak k$$
are stable under $Ad_K$ in $\mathfrak g$, where $\mathfrak k = Lie~K$, and so are Lie subalgebras.  This implies $UK$ and $TK$ are subgroups of $G$, hence  stable under conjugation by $K$, and thus their orbits on $G/K$ are stable under the left-$K$ action.  That is, $K$ preserves the Riemannian product $T\times U$.

Finally, as $S=T\times U$ and $G=SK$, we see that $G$ preserves the Riemannian product $T\times U$ which proves Proposition \ref{prop: isom of solvsoliton splits}.

\section{Applications to compact quotients of solvmanifolds}
\label{sec: applications to compact quot of solv}

\subsection*{Proof of Theorem \ref{thm: general finite volume quotient is symmetric}}
It is well-known that $(iv) \Longrightarrow   (i) \Longrightarrow (ii) \Longrightarrow (iii)$.  We will show that $(iii)\Longrightarrow(iv)$.

We begin by observing that $\mathfrak s$ is admissible as the hypothesis of positivity  implies $\mathfrak{z(s)} = 0$.  Notice also that the nilradical of $\mathfrak s$ is $\mathfrak{n(s)} = [\mathfrak s,\mathfrak s]$ as the positive element  of $\mathfrak s$ is non-singular on $\mathfrak{n(s)}$.

\begin{lemma} Let $\mathfrak s$ be an almost completely solvable, positive  Lie algebra with inner product and $\mathfrak s'$ the standard modification.  Then $\mathfrak s'$ is also almost completely solvable and positive.
\end{lemma}

\begin{proof}
The fact that $\mathfrak s'$ is almost completely solvable is the content of Theorem \ref{thm: std. mod. preserves almost compl. solv.}.

Recall, the standard modification is a normal modification (Prop.~\ref{prop: properties of std. mod.}) and so $\mathfrak s'$ is spanned by $[\mathfrak s,\mathfrak s]$ together with $\{Y_1 + D_1, \dots , Y_k + D_k \}$, where $\{Y_1, \dots, Y_k\}$ form a basis of any complement $\beta$ to $[\mathfrak s,\mathfrak s]$ in $\mathfrak s$ and $D_i\in N_l(\mathfrak s)$.  Here $N_l(\mathfrak s)$ denotes the skew-symmetric derivations relative to the given inner product.

Recall that any derivation $D\in Der(\mathfrak s)$ preserves $[\mathfrak s,\mathfrak s]$.  As $N_l(\mathfrak s)$ is the Lie algebra of  a compact group, it is fully reducible and so there exists a complement $\beta$ to $[\mathfrak s,\mathfrak s]$ in $\mathfrak s$ which is preserved by $N_l(\mathfrak s)$. Finally, since every derivation is valued in $\mathfrak{n(s)} = [\mathfrak s,\mathfrak s]$ (Lemma  \ref{lemma: derivation of solvabe goes to nilradical}), we see that $\beta \subset Ker~D$, for all $D\in N_l(\mathfrak s)$.

Next we show that $\mathfrak{n(s')} = \mathfrak{n(s)}$. 
Recall, the nilradical of a solvable Lie algebra is the collection of nilpotent elements.  As $[\mathfrak s,\mathfrak s] \subset \mathfrak s'$ consists of nilpotent elements, we consider $Y+D\in \mathfrak s'$ where $Y\in\beta$ and $D$.  Assume $Y+D$ is nilpotent.  As $Y\in Ker~D$, $ad~Y$ and $D$ commute and so the generalized eigenspaces of $ad~Y$ are preserved by $D$.  This implies the  eigenvalues of $Y+D$ are sums of those from $ad~Y$ and $D$.  However, $Y+D$ being nilpotent means that zero is the sole generalized eigenvalue and almost complete solvability implies $D=0$.  Thus $Y$ is nilpotent.  This is a contradiction since $\beta$ is a complement of $\mathfrak{n(s)}$.  This proves $\mathfrak{n(s')} = \mathfrak{n(s)}$.

Let $X$ be a positive element of $\mathfrak s$.  Observing that Lie's theorem implies the  eigenvalues of $ad~X$ are the same as those for $ad~(X+N) = ad~X+ad~N$, for any $N\in\mathfrak{n(s)}$ (Lemma \ref{lemma: eigenvalues of ad X = those of ad X+N for N nilpotent}),  we see that there exists a positive element in $\beta$.  From the above, we see that there exists $D\in N_l(\mathfrak s)$ such that $X+D\in \mathfrak s'$ and $ad~X$ commutes with $D$.  Now $D$ preserves the generalized eigenspaces of $ad~X$ and we see that $ad~X+D$ has  eigenvalues with positive real part on $\mathfrak{n(s')}=\mathfrak{n(s)}$.  That is, $\mathfrak s'$ is positive.
\end{proof}

Using this lemma,  we may assume that $\mathfrak s$ is in standard position and, by Lemma \ref{lemma: s2 properties}, $\mathfrak s$ admits an $LR$-decomposition  $\mathfrak s = \mathfrak s_1 + \mathfrak s_2$ with $\mathfrak s_2 < \mathfrak g_2$ an ideal of the full isometry algebra $\mathfrak g$.   
Now, for $X\in\mathfrak n (\mathfrak s_2)$,  if $ad~X|_{\mathfrak s_2} ^k = 0$ then $ad~X^{k+1}=0$.  This shows $\mathfrak n ( \mathfrak s_2) \subset \mathfrak{n(s)}$.

\begin{lemma} If non-trivial, the ideal $\mathfrak s_2$ is positive.
\end{lemma}

\begin{proof}
	Consider the adjoint action of $\mathfrak g_1$ on the ideal $\mathfrak s_2$.   As $\mathfrak g_1$ is semi-simple, this representation is fully reducible.  Thus, there exists an invariant decomposition $\mathfrak s_2 = \beta + \mathfrak n(\mathfrak s_2)$.  Since derivations of solvable Lie algebras take their image in the nilradical (Lemma \ref{lemma: derivation of solvabe goes to nilradical}), we see that $\mathfrak g_1$ commutes with $\beta$.
	
	Now take $X\in \mathfrak s$ which is a positive element and write it as $X=X_1+X_2 \in \mathfrak s_1  +\mathfrak s_2$.  As positivity does not change upon adding an element of the nilradical, by the above lemma we may assume that $X_2\in \beta$.  
Since $\mathfrak g_1$ commutes with $X_2$, we have that $ad~\mathfrak g_1|_{\mathfrak s_2}$ commutes with $ad~X_2|_{\mathfrak s_2}$ and so preserves the generalized eigenspaces $V_\lambda$ of $ad~X_2|_{\mathfrak s_2}$.
		
	Observe that $ad~X_1|_{V_\lambda}$ is an element of a representation of the  semi-simple algebra $ad~\mathfrak g_1|_{V_\lambda}$; hence $tr \ ad~X_1|_{V_\lambda} =0$.  The space $V_\lambda$ is preserved by $ad~X$ and so the generalized eigenvalues of $ad~X$ restricted to $V_\lambda$ still have positive real part.  Now we see that
		$$ 0 < tr \ ad~X|_{V_\lambda}  = tr \ ad~X_2|_{V_\lambda} = \mathbb R(\lambda) \cdot \dim V_\lambda  $$
This shows that $\mathfrak s_2$ is positive with positive element $X_2$.	
\end{proof}

\begin{lemma} If $\mathfrak s_2$ is non-trivial, $\mathfrak g$ is non-unimdular.
\end{lemma}

\begin{proof}
Take $X\in \mathfrak s_2$ which is a positive element.  As $\mathfrak s_2$ is ideal (Lemma \ref{lemma: s2 properties}), we see that $tr_\mathfrak g \  ad~X = tr_{\mathfrak s_2} \ ad~X > 0$.
\end{proof}

Finally, if $\mathfrak g$ is unimodular, we see that $\mathfrak s_2$ is trivial which implies $\mathfrak g = \mathfrak g_1$ is semi-simple.  By Theorem \ref{thm: linear isom grp} and Corollary \ref{cor: stabilizer respects levi decomposition}, we see that $G$ is linear and the maximal compact subgroup $K$ fixes some point.  Thus, $S$ is isometric to $G_{nc}/K_{nc}$ with some $G_{nc}$-invariant metric, where $K_{nc}$ is a maximal compact subgroup of $G_{nc}$.  By \cite[Ch.~IV, Prop.~3.4]{Helgason},   $S$ is in fact a symmetric space.  This proves Theorem \ref{thm: general finite volume quotient is symmetric}.

\begin{proof}[Proof of Corollary  \ref{cor: negative curv rigidity}]
	Let $G/H$ be a homogeneous space which admits a $G$-invariant metric of negative curvature.  For the moment, we endow $G/H$ with a metric of negative curvature.  From \cite{HeintzeHMNC}, we see  that there exists a solvable group $S$ of isometries such that
		\begin{quote}
		\begin{enumerate}
		\item The codimension of $[\mathfrak s,\mathfrak s]$ is 1.
		\item The algebra $\mathfrak s$ is positive.
		\end{enumerate}
\end{quote}		
An algebra satisfying these two properties is called a negative curvature algebra.

From this, we immediately see that $\mathfrak s$ is almost completely solvable and admissible (Proposition \ref{prop: neg curv alg is almost compl solv}).  Therefore, the transitive group $G$ contains a transitive solvable subgroup $R$ (Theorem \ref{thm: linear isom grp} and Proposition \ref{prop: linear isom implie strongly solvable}).  This solvable group can be assumed to be acting simply transitively and, again by the work of Heintze \cite{HeintzeHMNC}, its Lie algebra is a negative curvature algebra.  Thus, $R$ is almost completely solvable and positive.  

As any $G$-invariant metric on $G/H$ gives rise to a left-invariant metric on $R$,  Theorem \ref{thm: general finite volume quotient is symmetric} applies and our result holds. 
\end{proof}

\begin{proof}[Proof of Corollary \ref{cor: einstein alg rigidity}]
	This corollary follows immediate from the theorem and the classification results of Lauret \cite{LauretStandard}.  There one may quickly deduce that an algebra admitting an Einstein metric is indeed almost completely solvable and positive.
\end{proof}

\providecommand{\bysame}{\leavevmode\hbox to3em{\hrulefill}\thinspace}
\providecommand{\MR}{\relax\ifhmode\unskip\space\fi MR }
\providecommand{\MRhref}[2]{%
  \href{http://www.ams.org/mathscinet-getitem?mr=#1}{#2}
}
\providecommand{\href}[2]{#2}

\end{document}